\documentclass[12pt]{amsart}
\usepackage{amssymb,amsmath,amsthm,amssymb}
\usepackage{mathtools}
\usepackage{graphicx}
\usepackage{comment}
\usepackage{yhmath}
\usepackage[hypertexnames=false]{hyperref}
\usepackage{enumerate}
\usepackage{color}
\usepackage{epstopdf}
\usepackage{graphicx,scalerel}
\graphicspath{{figures/}}
\usepackage{tikz}
\usepackage{mathrsfs}
\usepackage{float}
\usetikzlibrary{decorations.pathreplacing,calligraphy,arrows.meta,angles,quotes}

\renewcommand{\epsilon}{\varepsilon}

\DeclareMathOperator{\arccosh}{arccosh}

\def\R{\mathbb{R}}

\def\tilde{\widetilde}

\def\Si{\Sigma}

\def\la{\langle}
\def\ra{\rangle}

\def\S{\mathbb{S}}
\def\H{\mathbb{H}}

\def\ba #1\ea {\begin{align} #1\end{align}}
\def\bann #1\eann {\begin{align*} #1\end{align*}}
\def\ben #1\een {\begin{enumerate} #1\end{enumerate}}
\def\bi #1\ei {\begin{itemize}\renewcommand\labelitemi{--} #1\end{itemize}}

\newtheorem{theorem}{Theorem}[section]
\newtheorem*{theorem*}{Theorem}
\newtheorem{lemma}[theorem]{Lemma}
\newtheorem{proposition}[theorem]{Proposition}
\newtheorem{remark}[theorem]{Remark}
\newtheorem{example}[theorem]{Example}
\newtheorem{corollary}[theorem]{Corollary}
\newtheorem{definition}[theorem]{Definition}

\newtheorem*{conjecture*}{Conjecture}

\newtheorem*{claim*}{Claim}

\newtheoremstyle{TheoremNum}
        {\topsep}{\topsep}              
        {\itshape}                      
        {}                              
        {\bfseries}                     
        {.}                             
        { }                             
        {\thmname{#1}\thmnote{ \bfseries #3}}
    \theoremstyle{TheoremNum}

\author{Theodora Bourni}
\author{Jose M. Espinar}
\author{Aakash Mishra}
\address{Department of Mathematics, University of Tennessee Knoxville, Knoxville TN, 37996-1320}
\email{tbourni@utk.edu}
\address{Departamento de Geometr\'{i}a y Topolog\'{i}a $\&$ IMAG, University of Granada, Spain, 18071}
\email{jespinar@ugr.es}
\address{Department of Mathematics, University of Tennessee Knoxville, Knoxville TN, 37996-1320}
\email{amishra5@vols.utk.edu}

\date{\today}

\title[Aleksandrov reflection for geometric flows in $\mathbb{H}^{n+1}$]{Aleksandrov reflection for geometric flows in hyperbolic spaces}

\begin{document}

\begin{abstract}
We develop an Aleksandrov reflection framework for a large class of expanding curvature flows in hyperbolic space, with inverse mean curvature flow serving as a model case. The method applies to the level-set formulation of the flow, and as a consequence we obtain graphical and Lipschitz estimates. Using these estimates, we show that solutions become star-shaped and therefore converge exponentially fast to an umbilic hypersurface at infinity.
We also extend these results to the non-compact setting in two cases. First, assuming the asymptotic boundary of the solution consists of a single point, we show that the flow becomes a graph over a horosphere with uniform gradient bounds and converges to a limiting horosphere. Second, assuming the asymptotic boundary consists of two points, we prove that the flow eventually becomes a global graph over a hyperbolic cylinder with uniform gradient bounds; this is achieved through an explicit cylindrical barrier construction analogous to the horospherical one.
\end{abstract}

\maketitle

\setcounter{tocdepth}{1}

\section{Introduction}

Inverse mean curvature flow (IMCF) has played a central role in geometric analysis since the work of Huisken and Ilmanen, who used a weak formulation of the flow in Euclidean space to prove the Riemannian Penrose inequality \cite{HuiskenIlmanen2001}. Their construction, based on an elliptic regularisation of the level set formulation, produces a weak outward-minimising evolution which is well defined past singularities and enjoys powerful monotonicity properties. In parallel, Gerhardt and others developed a systematic theory of smooth inverse curvature flows for star-shaped and mean convex hypersurfaces, both in Euclidean and hyperbolic space \cite{Gerhardt2014}.

In hyperbolic space, inverse curvature flows and their variants have been used extensively to derive sharp geometric inequalities of Penrose and Aleksandrov--Fenchel   type. For instance, Andrews and Wei, and Andrews, Chen and Wei, used curvature flows preserving quermassintegrals or volume to prove Aleksandrov--Fenchel type inequalities for closed hypersurfaces in $\mathbb{H}^{n+1}$ \cite{AndrewsChenWei2021,AndrewsWei2018}, while Hu, Li and Wei obtained further hyperbolic inequalities via locally constrained curvature flows \cite{HuLiWei2022}. In dimension three, Hung and Wang revisited IMCF in $\mathbb{H}^3$ and clarified the relation between the flow and various notions of mass \cite{HungWang2015}. On the asymptotically hyperbolic side, Brendle, Hung and Wang, and de Lima and Gir{\~a}o, established sharp Minkowski and Penrose type inequalities using IMCF and related flows in the anti-de Sitter--Schwarzschild setting and for asymptotically hyperbolic graphs \cite{BrendleHungWang2016, DeLimaGirao2016}. Altogether, these works show that expanding curvature flows in negatively curved ambient spaces provide a flexible tool for encoding global geometric information.

A different but closely related perspective on expanding flows is provided by the Aleksandrov reflection method. In Euclidean space, Chow and Gulliver introduced a reflection scheme for extrinsic curvature flows that systematically exploits symmetry with respect to affine hyperplanes and yields strong control on the geometry of evolving hypersurfaces \cite{ChowGulliver01}. Their ideas have since been developed in many directions and are surveyed in detail by Chow \cite{Chow2023}. In particular, Aleksandrov reflection has proved especially effective for expanding and ancient solutions: it can be used to propagate symmetries along the flow and to obtain rigidity and classification results under relatively weak hypotheses on the initial data. This is exploited, for example, in the work of Bourni, Langford and Tinaglia~\cite{BourniLangfordTinaglia2020, BourniLangfordTinaglia2021} and in the classification of convex ancient solutions to mean curvature flow by Brendle and Choi~\cite{BrendleChoi2019, BrendleChoi2021}.

In hyperbolic space, the Aleksandrov reflection method for \emph{stationary} (CMC and Weingarten) hypersurfaces in hyperbolic space goes back to do Carmo--Lawson~\cite{doCarmo} and to Levitt--Rosenberg~\cite{LevittRosenberg1985}, with later refinements by Sa Earp--Toubiana~\cite{SaEarpToubiana2001}. The present paper develops the parabolic counterpart of this scheme for the level-set formulation of expanding curvature flows, both in the compact case and in the non-compact case with one or two points at infinity.

Building on the weak IMCF framework of Huisken and Ilmanen and on estimates for smooth star-shaped flows, Harvie studied weak inverse mean curvature flows in $\mathbb{H}^{n+1}$ (and, earlier, in Euclidean space \cite{harvie2022IMCFnonstarshaped} for non-star-shaped initial data). He showed that after a finite time these hyperbolic flows become star-shaped and, in low dimensions, smooth, so that they subsequently evolve by the classical IMCF. As an application, he obtained Minkowski-type inequalities for outer-minimizing domains in hyperbolic space \cite{Harvie2024}, previously known only for star-shaped hypersurfaces \cite{BrendleHungWang2016, DeLimaGirao2016}. His work can be viewed as a first hyperbolic counterpart to the Euclidean Aleksandrov reflection theory for extrinsic flows developed by Chow and Gulliver. A complementary general construction of weak IMCF in ambient spaces with Ricci $\ge -(n-1)$ was given by Mari, Rigoli and Setti~\cite{MariRigoliSetti2022} using a $p$-Laplace approximation; their framework covers $\mathbb H^{n+1}$ in every dimension.

The present paper continues this program by studying expanding curvature flows in $\mathbb{H}^{n+1}$ from the dual perspective of weak formulations and Aleksandrov reflection. Throughout the paper we study a class of expanding extrinsic curvature flows of properly embedded hypersurfaces in $\H^{n+1}$. 
At the parametric level, if $\Sigma$ is a complete hypersurface and 
$\phi_t : \Sigma \to \H^{n+1}$, $t\in[0,T)$, is a one-parameter family of proper embeddings with inward unit normal $\eta_t$ and
principal curvatures $\kappa_1\leq \dots \leq \kappa_n$, the flow is formally given by
\begin{equation*}
    \frac{\partial \phi_t}{\partial t} \;=\;  F(\kappa_1,\dots,\kappa_n;t)\,\eta_t \,.
\end{equation*}
The structural assumptions on the speed $F$ and the precise class of flows covered in this work are collected in
Section~\ref{sec:levelset}. As a basic example, the inverse mean curvature flow corresponds to the choice
$F(\kappa_1,\dots,\kappa_n;t) = -1/H_t$, where $H_t = \kappa_1+\cdots+\kappa_n$ denotes the mean curvature of $\Sigma_t = \phi_t(\Sigma)$.

We work primarily with weak, viscosity-type solutions in the spirit of Evans--Spruck and Chen--Giga--Goto \cite{CGG91,EvansSpruck91}, suitably adapted to the hyperbolic setting. We show that the interplay between the level set formulation and reflection across totally geodesic hyperplanes yields strong structural information on the evolving hypersurfaces. In particular, we prove that appropriate weak solutions become star-shaped after a finite waiting time, which we compute explicitly. If, in addition, the flow is smooth, then classical results describe its subsequent evolution. It is reasonable to conjecture that one does not need to assume a priori smoothness of the flow; rather, once the flow becomes star-shaped, it automatically becomes smooth. This phenomenon is indeed observed, for instance, for the inverse mean curvature flow in low dimensions \cite{LiWei2017, Harvie2024}.

Moreover, we establish uniform geometric control that allows us to analyse the long-time behaviour of the flow and to describe its asymptotic geometry. 
From a broader viewpoint, our results fit into a growing body of work that uses expanding curvature flows in hyperbolic space to derive global geometric information and sharp inequalities. While previous approaches have largely relied on star-shapedness or strong convexity assumptions preserved along smooth flows \cite{AndrewsChenWei2021,AndrewsWei2018,Gerhardt2014,HuLiWei2022}, our use of a weak viscosity formulation combined with Aleksandrov reflection allows us to treat more general initial data, including outer-minimising hypersurfaces. A novel aspect of our work is that we also allow non-compact initial data. Very little is known in the non-compact setting compared to the compact case, and Aleksandrov reflection is typically much more delicate in this context. Nevertheless, we develop an Aleksandrov reflection scheme that works for these non-compact configurations.  

In this way, the hyperbolic inverse mean curvature flow inherits some of the flexibility of the weak Euclidean theory of Huisken--Ilmanen, while retaining enough regularity, after a finite waiting time, to admit a detailed asymptotic analysis.

The main results of this paper are threefold. 
First, we adapt Aleksandrov reflection to the level set formulation of expanding curvature flows in hyperbolic space and derive graphical and Lipschitz estimates for viscosity solutions that remain outside a fixed geodesic ball (Theorems~\ref{thm: Lipschitz bound} and~\ref{thm: gradient bound}). 
Second, for inverse curvature flows with general admissible speed functions \(F\), we combine these estimates with Harvie's weak IMCF theory \cite{Harvie2024} and Gerhardt's smooth analysis \cite{Gerhardt2014} to perform an asymptotic analysis (Theorem~\ref{thm: Reg}). This shows that weak solutions become star-shaped, and if the solution is additionally smooth, it converges exponentially fast to an umbilic hypersurface at infinity.\\  
Third, in the non-compact setting of properly embedded hypersurfaces whose asymptotic boundary consists of exactly one or two points, we introduce a horospherical (resp.\ cylindrical) version of Aleksandrov reflection. This yields global graphical representations and gradient bounds over a horosphere (Theorem~\ref{thm:gradup}) and over a cylinder (Theorem~\ref{thm:grad2pts}), respectively. In the one-point case we further show that, for a large class of curvature functions, the property of lying between two horospheres is preserved along the flow; this is achieved through an explicit construction of barrier hypersurfaces (Theorem~\ref{thm:upbar}). For the inverse mean curvature flow we combine these estimates with Allen's asymptotic theory~\cite{Allen} to deduce long-time existence and convergence to a limiting horosphere (Theorem~\ref{thm:NonCompact}). In the two-point case we obtain a cylindrical analogue of the spherical-barrier argument: if $\Sigma_0$ lies outside a hyperbolic cylinder $\mathcal C(s_-)$, then $\Sigma_t$ stays outside an explicit expanding cylinder $\mathcal C(s_-(t))$ with $\sinh s_-(t)=e^{t/n}\sinh s_-$ (Theorem~\ref{thm:cylbar}); consequently, $\Sigma_t$ becomes a global graph over the cylinder $\mathcal C(\bar s)$ after a finite waiting time $T=n\log(\sinh\bar s/\sinh s_-)$ (Corollary~\ref{cor:graph-after-T} and Theorem~\ref{thm:NonCompact2pts}). Long-time existence of the two-point flow and the existence of an asymptotic profile remain open; the obstruction is discussed in Remark~\ref{rem:open-2pt}.

\section{Preliminaries}\label{sec: 2}

In this section, we recall the necessary background for the rest of the paper.

\subsection{Hyperbolic space}

We denote by \(\mathbb{H}^{n+1}\) the simply connected, complete Riemannian manifold of constant sectional curvature \(-1\), endowed with the hyperbolic metric \(g_{\mathbb{H}^{n+1}}\). The Euclidean metric on \(\mathbb{R}^{n+1}\) will be denoted by \(g_{\mathbb{R}^{n+1}}\); we use \(\langle \cdot,\cdot\rangle\) and \(|\cdot|\) for the Euclidean scalar product and norm, and \(\langle \cdot,\cdot\rangle_{\mathbb{H}^{n+1}}\), \(|\cdot|_{\mathbb{H}^{n+1}}\) for the corresponding hyperbolic quantities. We will omit the subindices if no confusion occurs.

Fix a point \(p_0\in\mathbb{H}^{n+1}\). In geodesic polar coordinates around \(p_0\), the exponential map
\[
\exp_{p_0} : (0,\infty)\times\mathbb{S}^n \to \mathbb{H}^{n+1}, \qquad (r,\theta)\mapsto \exp_{p_0}(r\theta),
\]
identifies \(\mathbb{H}^{n+1}\setminus\{p_0\}\) with \((0,\infty)\times\mathbb{S}^n\). In these coordinates the hyperbolic metric takes the form
\begin{equation}\label{eq:polar-metric}
d\bar{s}^2 \;=\; dr^2 + \sinh^2(r)\,\sigma_{ij}\,dx^i dx^j,
\end{equation}
where \(\sigma_{ij}\) is the standard round metric on \(\mathbb{S}^n\). Here \(r\) is the hyperbolic distance to \(p_0\). This description will be used repeatedly when we consider geodesic spheres and radial flows.

\medskip

The hyperbolic space admits a natural compactification by adding its \emph{boundary at infinity}. Two geodesic rays \(\gamma_1,\gamma_2:[0,\infty)\to\mathbb{H}^{n+1}\) are asymptotic if their hyperbolic distance remains uniformly bounded. The set of equivalence classes of geodesic rays is denoted by \(\partial_\infty\mathbb{H}^{n+1}\) and is homeomorphic to \(\mathbb{S}^n\). We write
\[
\overline{\mathbb{H}^{n+1}} \;=\; \mathbb{H}^{n+1}\cup\mathbb{S}^n_\infty,
\]
where \(\mathbb{S}^n_\infty:=\partial_\infty\mathbb{H}^{n+1}\) and $\overline{\mathbb{H}^{n+1}}$  can be endowed with a natural topology, the {\it cone topology}, which makes $\overline{\mathbb{H}^{n+1}}$ homeomorphic to a closed topological ball. If \(\Sigma\subset\mathbb{H}^{n+1}\) is a properly embedded hypersurface, its \emph{asymptotic boundary} is
\[
\partial_\infty \Sigma \;:=\; \overline{\Sigma}\cap\mathbb{S}^n_\infty,
\]
where the closure is taken in the cone topology on \(\overline{\mathbb{H}^{n+1}}\). This notion will be crucial when studying non-compact solutions and their asymptotic behavior.

\medskip

A complete connected hypersurface \(\Sigma\subset\mathbb{H}^{n+1}\) is called \emph{isoparametric} if all of its principal curvatures are constant \cite{CecilRyan}. In hyperbolic space, isoparametric hypersurfaces are necessarily totally umbilic, and up to ambient isometries they fall into four families:
\begin{itemize}
\item \emph{Geodesic spheres} \(\mathcal S_r(p_0)\) of radius \(r>0\) centered at a point \(p_0\). With respect to the inward normal, all principal curvatures are equal to \(\coth r\), so their mean curvature is
\[
H = n\,\coth r.
\]
\item \emph{Totally geodesic hyperplanes}, for which all principal curvatures vanish and hence \(H\equiv 0\).
\item \emph{Equidistant hypersurfaces} at signed distance \(s\in\mathbb{R}\) from a totally geodesic hyperplane. With a suitable choice of unit normal, all principal curvatures are equal to \(\tanh s\), so
\[
H = n\,\tanh s.
\]
\item \emph{Horospheres}, which are limits of geodesic spheres as the center tends to infinity, with the spheres passing through a fixed point of \(\mathbb H^{n+1}\). With the normal pointing into the corresponding horoball, all principal curvatures are identically \(1\), so
\[
H = n.
\]
\end{itemize}
These four classes will serve as model hypersurfaces and barriers along this paper.

\medskip

We now recall the two concrete models of \(\mathbb{H}^{n+1}\) used in the sequel and describe how the isoparametric families appear in each of them.

\medskip\noindent
\textbf{Poincar\'e ball model.}
In this model, \(\mathbb{H}^{n+1}\) is identified with the unit ball
\[
\mathbb{B}^{n+1} :=  \{ x\in\mathbb{R}^{n+1} : |x|<1\}
\]
equipped with the metric
\begin{equation}\label{eq:ball-metric}
g_{\mathbb{H}^{n+1}} \;=\; \frac{4}{(1-|x|^2)^2}\,g_{\mathbb{R}^{n+1}}.
\end{equation}
Geodesics correspond to Euclidean line segments through the origin and circular arcs that meet \(\partial\mathbb{B}^{n+1}=\mathbb{S}^n\) orthogonally. The boundary at infinity is identified with the Euclidean unit sphere:
\[
\partial_\infty\mathbb{H}^{n+1} \;\cong\; \mathbb{S}^n = \partial\mathbb{B}^{n+1}.
\]

From the Euclidean point of view, the isoparametric families can be described as follows:
\begin{itemize}
\item Geodesic spheres centered at the origin correspond to Euclidean spheres \(\{x\in\mathbb{B}^{n+1}: |x|=\rho\}\), with \(\rho\in(0,1)\) related to the hyperbolic radius \(r\) by
\begin{equation}\label{rho}
\rho \;=\; \tanh\!\left(\frac{r}{2}\right).
\end{equation}
More generally, geodesic spheres centered at arbitrary points of \(\mathbb{H}^{n+1}\) correspond to Euclidean spheres contained in \(\mathbb{B}^{n+1}\).
\item Totally geodesic hyperplanes are either Euclidean hyperplanes through the origin, or Euclidean hyperspheres meeting \(\mathbb{S}^n\) orthogonally.

\item Fix a totally geodesic hyperplane \(P\). The hypersurfaces at constant signed hyperbolic distance from \(P\)
form a coaxial pencil of \emph{Euclidean spheres \emph{and Euclidean hyperplanes}} in \(\mathbb{B}^{n+1}\),
all of them intersecting \(\mathbb{S}^n\) \emph{with constant angle} \(\alpha\in(0,\pi/2)\) (hence not orthogonally).
Equivalently, they are precisely the Euclidean spheres/hyperplanes whose ideal boundary in \(\mathbb{S}^n\)
coincides with \(\partial_\infty P\); the Euclidean hyperplane appears as the limiting member of the pencil
(a sphere of infinite radius).
\item Horospheres are Euclidean spheres tangent to \(\mathbb{S}^n\) from the inside. The point of tangency is the corresponding point of \(\partial_\infty\mathbb{H}^{n+1}\), and the interior of the Euclidean sphere represents the associated horoball.
\end{itemize}
This Euclidean picture will be used repeatedly in the Aleksandrov reflection arguments in the ball model.

\medskip\noindent
\textbf{Upper half-space model.}
In this model, \(\mathbb{H}^{n+1}\) is represented as
\[
\mathbb{R}^{n+1}_+ := \{(x,y)\in\mathbb{R}^n\times(0,\infty)\},
\]
with the metric
\begin{equation}\label{eq:halfspace-metric}
g_{\mathbb{H}^{n+1}} \;=\; \frac{1}{y^2}\,g_{\mathbb{R}^{n+1}}.
\end{equation}
The boundary at infinity is identified with
\[
\partial_\infty\mathbb{H}^{n+1} \;\cong\; \big(\mathbb{R}^n\times\{0\}\big)\,\cup\,\{\infty\},
\]
which is again homeomorphic to \(\mathbb{S}^n\). In these coordinates:
\begin{itemize}
\item Geodesic spheres correspond to Euclidean spheres contained in \(\mathbb{R}^{n+1}_+\).
\item Horospheres are given by horizontal slices \(\{y=c\}\), \(c>0\), and Euclidean spheres tangent to $\{y=0\}$.
\item Totally geodesic hyperplanes are either vertical hyperplanes or Euclidean hemispheres orthogonal to \(\{y=0\}\).
\item Equidistant hypersurfaces correspond to either Euclidean hyperplanes  or Euclidean spheres that intersect \(\{y=0\}\) \emph{with a constant angle}, \(\alpha\in(0,\pi/2)\).
\end{itemize}
The two models are related by explicit conformal diffeomorphisms which extend continuously to the boundary at infinity, and we will move freely between the two depending on the context.\\

\subsection{Aleksandrov reflection in hyperbolic space}~\label{sec:Aleksandrov}

The Aleksandrov reflection method in \(\mathbb{H}^{n+1}\) that we use below is the hyperbolic analogue of the reflection scheme for viscosity solutions of extrinsic curvature flows in \(\mathbb{R}^{n+1}\) developed by Chow and Gulliver \cite{ChowGulliver01}. We recall here the model-independent formulation and its concrete realizations in the Poincar\'e ball and upper half-space models. This method will be applied both to stationary hypersurfaces and to level-set solutions of expanding curvature flows.

Let \(P\subset\mathbb{H}^{n+1}\) a totally geodesic hyperplane. The complement \(\mathbb{H}^{n+1}\setminus P\) has two connected components, which we denote by \(H_+(P)\) and \(H_-(P)\). There exists an isometric involution 
\[
\mathcal R_P : \mathbb{H}^{n+1} \to \mathbb{H}^{n+1}
\]
fixing \(P\) pointwise and exchanging the two half-spaces \(H_+(P)\) and \(H_-(P)\); we call \(\mathcal R_P \in \mathrm{Iso}(\mathbb{H}^{n+1}) \) the \emph{reflection} across \(P\). Every such reflection extends continuously to an involution of \(\overline{\mathbb{H}^{n+1}}\), acting as a conformal diffeomorphism of \(\mathbb{S}^n_\infty\).

Let \(\Sigma\subset\mathbb{H}^{n+1}\) be a properly embedded, connected \(C^2\) hypersurface. For a one-parameter family of totally geodesic hyperplanes \(\{P(s)\}_{s\in I}\), we denote by \(\mathcal R _s := \mathcal R_{P(s)}\) the corresponding family of reflections and by \(H_\pm(s):=H_\pm(P(s))\) the associated half-spaces. The Aleksandrov reflection method consists in:
\begin{itemize}
\item starting with \(P(s)\) far away from \(\Sigma\) so that \(\Sigma\cap H_+(s)=\emptyset\);
\item moving the hyperplanes monotonically until the first value \(s_0\in I\) where $P(s_0)$ touches $\Sigma$ for the first time;
\item continue reflecting the portion \(\mathcal R_{s}(\Sigma\cap H_+(s))\) monotonically until the first value \(s_1\in I\) where reflected portion \(\mathcal R_{s_1}(\Sigma\cap H_+(s_1))\) touches \(\Sigma\);
\end{itemize}
When the hypersurface satisfies a suitable elliptic PDE, or evolves according to a suitable parabolic PDE, this method can be used to obtain reflection symmetry; see, e.g.,~\cite{ChowGulliver01,Chow2023,Harvie2024,doCarmo}.

\section{Level set flow}\label{sec:levelset}

In this section we make precise the class of expanding curvature flows considered in this paper and recall their level-set
formulation in $\H^{n+1}$, following \cite{CGG91,ChowGulliver01,EvansSpruck91}.

Let $\Sigma$ be a compact embedded hypersurface in $\H^{n+1}$. For $t\in[0,T)$, let
$\phi_t : \Sigma \to \H^{n+1}$ be a one-parameter family of embeddings such that
$\Sigma_t := \phi_t(\Sigma)$ is a compact embedded $C^2$ hypersurface with inward unit normal $\eta_t$ and principal
curvatures $\kappa_1\leq \dots \leq \kappa_n$. The geometric evolution is formally given by
\begin{equation}\label{eveq-again}
    \frac{\partial \phi_t}{\partial t}
    \;=\;  F(\kappa_1,\dots,\kappa_n; t)\,\eta_t ,
\end{equation}
where $F$ is a real-valued function of the principal curvatures and of time $t\in[0,T)$, $0<T\leq\infty$. We assume:
\begin{enumerate}
    \item[(i)] For each $t\geq 0$, the map $F_t := F(\cdot,t)$ is $C^1$ in the curvature variables.
    \item[(ii)] $F$ is non-decreasing in each $\kappa_i$ (so that the evolution is degenerate parabolic).
    \item[(iii)] Our sign convention is such that each $\kappa_i$ is positive on geodesic spheres in $\H^{n+1}$ (with respect to the inward unit normal).
\end{enumerate}
We are primarily interested in expanding flows, which in this sign convention correspond to $F>0$ on strictly convex
hypersurfaces, that is,  when all the principal curvatures $\kappa_i$ are positive. The inverse mean curvature flow is included as the case $F(\kappa_1,\dots,\kappa_n;t)=-1/H_t$.

Even if $\Sigma$ is smooth and strictly mean convex ($H>0$), the flow \eqref{eveq-again} may develop singularities in finite time.
To treat such situations and allow for low-regularity initial sets we work in the level-set framework. In this approach, one considers the evolution of a closed set represented as the zero level set of
a continuous function solving a degenerate parabolic PDE determined by $F$.

Let $\Sigma_0$ be an embedded compact $C^0$ hypersurface bounding an open (not
necessarily connected) set $\Omega_0\subset\H^{n+1}$, so that $\partial\Omega_0=\Sigma_0$. Fix $K>0$
and define the truncated signed distance
\begin{equation*}
   u_0(p) \;:=\;
   \begin{cases}
      \min\{d_{\mathbb{H}^{n+1}}(p,\Sigma_0),K\}, & p\in \H^{n+1}\setminus \overline{\Omega}_0,\\[0.2em]
      -d_{\mathbb{H}^{n+1}}(p,\Sigma_0), & p\in \overline{\Omega}_0,
   \end{cases}
\end{equation*}where $d_{\mathbb{H}^{n+1}} (\cdot, \cdot)$ denotes the hyperbolic distance function. Let $u : \H^{n+1} \times [0,+\infty) \to \R$ be the level-set flow, i.e.\ the unique continuous viscosity solution of the
degenerate parabolic equation
\begin{equation}\label{eq:levelset}
    \frac{\partial u}{\partial t} \;=-\; \,|D u(\cdot, t)|_{\mathbb{H}^{n+1}}\,
    F(\kappa_1,\dots,\kappa_n,t),
\end{equation}
where $\kappa_1,\dots,\kappa_n$ are the eigenvalues of
$D(|Du|^{-1}Du)$ \emph{restricted to the orthogonal complement $(Du)^\perp$}, and $D$ denotes the Levi-Civita connection of $\mathbb{H}^{n+1}$, with initial condition $u(\cdot,0)=u_0(\cdot)$ (see~\cite{CGG91,EvansSpruck91} for this convention).
The evolving zero level set
\begin{equation}\label{Sigmat}
    \Sigma_t \;:=\; \{ x\in\H^{n+1} : u(x, t)=0\}
\end{equation}
is then called the \emph{generalized solution} of the evolution problem \eqref{eveq-again}. This generalized solution is
unique, and for $t>0$ the set $\Sigma_t$ depends only on $\Sigma_0$ and not on the particular choice of $u_0$.
Since $\Sigma_0$ is compact, each $\Sigma_t$ is a compact set for $t\in[0,\infty)$.

It is convenient to define $u_t(\cdot)=u(\cdot, t)$.
We also introduce the sublevel and superlevel sets
\begin{equation}\label{OEt}
\begin{split}
    \Omega_t &:= \{x\in\H^{n+1} : u_t(x)<0\},\\
    E_t &:= \{x\in\H^{n+1} : u_t(x)>0\}.
\end{split}
\end{equation}
In terms of $(\Sigma_t,\Omega_t,E_t)$ the Aleksandrov reflection method can be adapted to the level-set setting and used to obtain geometric estimates for the evolving hypersurfaces.

\medskip

\subsection{Admissible foliations (compact case)}
Fix the origin \({\bf 0}\in\mathbb{H}^{n+1}\), identified with the Euclidean origin in the Poincar\'e ball model \((\mathbb{B}_1({\bf 0}),g_{\mathbb{H}^{n+1}})\). For each unit vector \(\nu\in\mathbb{S}^n\subset T_0\mathbb{H}^{n+1}\), let
\(
\gamma_\nu:\mathbb{R}\to\mathbb{H}^{n+1}
\)
be the complete geodesic parametrized by arc-length with initial conditions \(\gamma_\nu(0)={\bf 0}\) and \(\dot{\gamma}_\nu(0)=\nu\). For every \(s\in\mathbb{R}\)  there is a unique totally geodesic hyperplane \(P_\nu(s)\) that is orthogonal to \(\gamma_\nu\) at \(\gamma_\nu(s)\). The family \(\{P_\nu(s)\}_{s\in\mathbb{R}}\) is a smooth foliation of \(\mathbb{H}^{n+1}\) by totally geodesic hyperplanes perpendicular to \(\gamma_\nu\).

We define the associated half-spaces
\begin{align*}
    H_+(\nu,s) &:= \bigcup_{\tilde s>s} P_\nu(\tilde{s}),\\
    H_-(\nu,s) &:= \bigcup_{\tilde s<s} P_\nu(\tilde{s}),
\end{align*}
so that \(\gamma_\nu((s,\infty))\subset H_+(\nu,s)\) and \(\gamma_\nu((-\infty,s))\subset H_-(\nu,s)\). The reflection across \(P_\nu(s)\) will be denoted by \(\mathcal R_{\nu,s}\in\mathrm{Iso}(\mathbb{H}^{n+1})\). 

In the Poincar\'e ball model, geodesics through the origin are Euclidean diameter:
\[
\gamma_\nu(s) = \tanh\!\left(\frac{s}{2}\right)\nu,\qquad s\in\mathbb{R}.
\]
For each \(s\in\mathbb{R}\), using the Euclidean scalar product, the hyperplane \(P_\nu(s)\) is represented by
\[
       P_\nu(s)
       =\Big\{x\in\mathbb{B}_1({\bf 0}):\langle x,\nu\rangle =\frac{\tanh(s)\,\big(1+|x|^2\big)}{2}\Big\},
\]
and the associated hyperbolic half-spaces are
\begin{equation}\label{eq:Hplus-ball}
       H_+(\nu,s)
       =\Big\{x\in\mathbb{B}_1({\bf 0}):\langle x,\nu\rangle>\frac{\tanh(s)\,\big(1+|x|^2\big)}{2}\Big\},
\end{equation}
and \(H_-(\nu,s)\) is given by the complementary strict inequality. This formula will be useful for encoding the half-space \(H_+(\nu,s)\) in the Poincar\'e ball.

In this model, the reflection \(\mathcal R_{\nu,s}\) across \(P_\nu(s)\) is the restriction of a M\"{o}bius transformation of \(\mathbb{B}_1({\bf 0})\). When $s=0$, it is realized as Euclidean reflection across a linear hyperplane through the origin orthogonal to $\nu$, and in the general case as Euclidean inversion across a sphere orthogonal to \(\partial\mathbb{B}_1({\bf 0})\) where the Euclidean center and radius are
\[
c_\nu(s) = \coth(s)\,\nu,\qquad R(s) = \sqrt{\coth^2(s)-1} = \frac{1}{\sinh(s)}.
\]

\medskip

\subsection{Optimal admissible values (compact case).}
 We now introduce the notion of admissible values, following Chow and Gulliver and our hyperbolic adaptation.

\begin{definition}\label{admissible}
Let \(\Sigma\) be a compact embedded hypersurface in \(\mathbb{H}^{n+1}\) and let \(\Omega\) be the bounded open set such that \(\partial\Omega=\Sigma\). Fix \(\nu\in\mathbb{S}^n\). A parameter \(s\in\mathbb{R}\) is called \emph{admissible in the direction \(\nu\)} (see Figure \ref{fig:placeholder}) if
\[
\mathcal R_{\nu,\tilde{s}}(\Sigma\cap H_+(\nu,\tilde{s}))\subset \overline{\Omega}\qquad\text{for all }\tilde{s}\in(s,\infty).
\]
The \emph{optimal admissible value} in the direction \(\nu\) is
\[
s_0(\nu) := \inf\{s\in\mathbb{R} : s \text{ is an admissible value in the direction }\nu\}.
\]
\end{definition}

\begin{figure}[!ht]
    \centering
\tikzset{every picture/.style={line width=0.75pt}} 
\begin{tikzpicture}[x=0.75pt,y=0.75pt,scale=0.9]

\draw   (200,150) .. controls (200,79.86) and (256.86,23) .. (327,23) .. controls (397.14,23) and (454,79.86) .. (454,150) .. controls (454,220.14) and (397.14,277) .. (327,277) .. controls (256.86,277) and (200,220.14) .. (200,150) -- cycle ;
\draw    (200,150) -- (454,150) ;
\draw   (274,135) .. controls (294,125) and (316,115) .. (364,135) .. controls (412,155) and (394,150) .. (396,170) .. controls (398,190) and (286,221) .. (265,184) .. controls (244,147) and (254,145) .. (274,135) -- cycle ;
\draw  [dash pattern={on 0.84pt off 2.51pt}]  (257,152) .. controls (270,187) and (349,196) .. (397,152) ;

\draw (461,140) node [anchor=north west][inner sep=0.75pt]   [align=left] {$\displaystyle P_{\nu }( s)$};
\draw (435,55) node [anchor=north west][inner sep=0.75pt]   [align=left] {$\displaystyle H_{+}( \nu ,s)$};
\draw (436,223) node [anchor=north west][inner sep=0.75pt]   [align=left] {$\displaystyle H_{-}( \nu ,s)$};
\draw (256,113) node [anchor=north west][inner sep=0.75pt]   [align=left] {$\displaystyle \Sigma $};
\draw (375,107) node [anchor=north west][inner sep=0.75pt]   [align=left] {$\displaystyle E$};
\draw (308,158) node [anchor=north west][inner sep=0.75pt]   [align=left] {$\displaystyle \Omega $};
\draw (272.25,180.99) node [anchor=north west][inner sep=0.75pt]  [font=\scriptsize,rotate=-0.01,xslant=-0.08] [align=left] {$\displaystyle R_{\nu ,s}( \Sigma \cap H_{+}( \nu ,s))$};
\end{tikzpicture}
    \caption{Admissibility of the value $s_0=0$}
    \label{fig:placeholder}
\end{figure}
We further define the \emph{overall optimal admissible value} by
\[
\bar{s} := \max\{ s_0(\nu) : \nu\in\mathbb{S}^n\}.
\]
Note that \(\bar{s}\geq 0\), since one always has \(s_0(\nu)+s_0(-\nu)\geq 0\) for all \(\nu\in\mathbb{S}^n\). 

In the level-set formulation, one can express admissibility directly in terms of the continuous function \(u:\mathbb{H}^{n+1}\to\mathbb{R}\) whose zero set describes the evolving hypersurfaces. 

\begin{definition}\label{admissible2}
    Let $\Sigma,\Omega,E$ be disjoint sets such that $\H^{n+1}=\Sigma\cup\Omega\cup E$. Given $\nu\in \S^n$ and $s\in\R$, we
    say that $s$ is \emph{admissible for the triple} $(\Sigma,\Omega,E)$ with respect to $\nu$ if
\begin{enumerate}
    \item[(A1)]\label{A1} 
        $\mathcal R_{\nu,s}(\Omega\cap H_+(\nu,s)) \subset \Omega \cap H_-(\nu,s)$, and 
    \item[(A2)]\label{A2} 
        $E\cap H_-(\nu,s) \subset \mathcal R_{\nu,s}(E \cap H_+(\nu,s))$.
\end{enumerate}
    Since $\H^{n+1}=\Sigma\cup\Omega\cup E$ is a disjoint union, these containments are equivalent to
\begin{enumerate}
    \item[(B1)]\label{B1}        
        $\mathcal R_{\nu,s}((\Omega\cup\Sigma)\cap H_+(\nu,s)) \subset (\Omega\cup\Sigma) \cap H_-(\nu,s)$, and
    \item[(B2)] \label{B2}
        $(E\cup\Sigma)\cap H_-(\nu,s) \subset \mathcal R_{\nu,s}((E\cup\Sigma) \cap H_+(\nu,s))$.
\end{enumerate}
\end{definition}

When $\Sigma$ is an embedded compact connected  hypersurface, $\Omega$ is the bounded open set with boundary
$\Sigma$, and $E$ is its complement, this definition agrees with the notion of admissibility introduced in the previous
section.

\begin{definition}\label{admissible3}
    Let $u:\H^{n+1}\to \R$ be a continuous function. Given $\nu\in \S^n$ and $s\in\R$, we say that $s$ is
    \emph{admissible for $u$ with respect to $\nu$} if
    \begin{equation}\label{equ}
        u(x)\;\geq\; u(x^*) \qquad \forall x\in H_+(\nu,s),
    \end{equation}
    where $x^* := \mathcal R_{\nu,s}(x)$.
\end{definition}

\begin{remark} We will also use the notions of the \emph{optimal admissible value} and the \emph{overall optimal admissible value}, as defined in Definition~\ref{admissible}, corresponding to the notions of admissibility as in Definitions \ref{admissible2} and \ref{admissible3}.
\end{remark}

\begin{lemma}\label{lem: admissibility}
    Let $\Sigma$ be the zero set of a continuous function $u$. If $s$ is admissible for $u$ with respect to $\nu$, then $s$ is
    admissible for the triple $(\Sigma,\Omega,E)$ with respect to $\nu$, where $\Omega=\{u<0\}$ and $E=\{u>0\}$. If $u$ is the
    signed distance function, then the converse is also true.
\end{lemma}

\begin{proof}
    Write $P:=P_\nu (s)$ and $H_{\pm} := H_{\pm}(\nu,s)$. If $s$ is admissible for $u$ with respect to $\nu$, then for all
    $x\in\Omega\cap H_{+}$ we have $u(x^*)\leq u(x)<0$, so $x^* \in \Omega\cap H_{-}$. Similarly, if $x\in E\cap H_-$, then
    $u(x^*)\geq u(x)>0$, which implies $x^*\in E\cap H_+$. Hence $s$ is admissible for the triple $(\Sigma,\Omega,E)$ with
    respect to $\nu$.

    Conversely, suppose that $s$ is admissible for the triple $(\Sigma,\Omega,E)$ with respect to $\nu$ and let $u$ be the
    signed distance to $\Sigma$, we need to show \eqref{equ}. 
    
    First, assume that $x\in (E\cup \Sigma)\cap H_+$. If $x^*\in \Omega \cup \Sigma$, then $u(x^*)\leq 0\leq u(x)$ and we are done. Consider now the case $x^*\in E$. Let $y\in\Sigma$ be a closest point to $x$ in $\Sigma$, so that $d_{\mathbb{H}^{n+1}}(x,y)=d_{\mathbb{H}^{n+1}}(x,\Sigma)$. Let $[x,x^*]$ denote the geodesic segment joining $x$ and $x^*$ that intersects $P$ orthogonally by construction. Hence, for any $p\in P$, $d_{\mathbb{H}^{n+1}}(x,p)=d_{\mathbb{H}^{n+1}}(x^*,p)$.
        \begin{itemize}
            \item If $y\in H_-$, the geodesic segment $[y,x]$ intersects $P$ at a unique point $p=P\cap [y,x]$. Then, by the triangle inequality, we get 
            \begin{align*}
            d_{\mathbb{H}^{n+1}}(x^*,y) &\le d_{\mathbb{H}^{n+1}}(x^*,p)+d_{\mathbb{H}^{n+1}}(p,y)\\
            & =d_{\mathbb{H}^{n+1}}(x,p)+d_{\mathbb{H}^{n+1}}(p,y)=d_{\mathbb{H}^{n+1}}(x,y),
            \end{align*}
            hence $u(x^*)\leq u(x)$.
    
            \item If instead $y\in H_+$, then $y^*\in (\Omega\cup\Sigma)\cap H_-$ and, using \eqref{B1}, we obtain
    \begin{align*}
        u(x)
        &= d_{\mathbb{H}^{n+1}}\big(x,(\Omega \cup \Sigma)\cap H_+\big)= d_{\mathbb{H}^{n+1}}\Big(x^*,\mathcal R_{\nu,s}\big((\Omega\cup\Sigma)\cap H_+\big)\Big)\\
        &\geq d_{\mathbb{H}^{n+1}}\big(x^*,(\Omega\cup\Sigma)\cap H_-\big)= u(x^*).
    \end{align*}
    \end{itemize}

    Next assume that $x \in (\Omega\cup\Sigma)\cap H_+$. In this case, \eqref{B2} implies
    \begin{align*}
        d_{\mathbb{H}^{n+1}}\big(x,(E\cup \Sigma)\cap H_-\big)
        &\geq d_{\mathbb{H}^{n+1}}\big(x^*,(E\cup \Sigma)\cap H_-\big)\\
        &\geq d_{\mathbb{H}^{n+1}}\Big(x^*,\mathcal R_{\nu,s}\big((E\cup \Sigma)\cap H_+\big)\Big)\\
        &= d_{\mathbb{H}^{n+1}}\big(x,(E\cup \Sigma)\cap H_+\big),
    \end{align*}
    and hence
    \begin{align*}
        u(x)
        &= -\,d_{\mathbb{H}^{n+1}}\big(x,(E\cup\Sigma)\cap H_+\big)\\
        &\geq -\,d_{\mathbb{H}^{n+1}}\big(x^*,(E\cup\Sigma)\cap H_-\big)= u(x^*).
    \end{align*}
    Therefore, $s$ is admissible for $u$ with respect to $\nu$.

\end{proof}

\medskip

\subsection{Optimal admissible value under the flow}

The next result shows that, under the flow \eqref{eveq-again}, the optimal admissible value does not deteriorate over time. Combined with the level-set formulation described above, this provides the hyperbolic analogue of the Euclidean viscosity reflection argument of Chow and Gulliver~\cite{ChowGulliver01}. 

Recall that throughout this section, our solutions -whether viscosity or strong- are assumed to be compact. Moreover, for each weak solution \(u\), there is an associated family of triplets \((\Sigma_t, \Omega_t, E_t)\) as defined in \eqref{Sigmat} and \eqref{OEt}.

\begin{theorem}\label{thm: admissible for all t}
Let $u$ be a solution to the level set flow given by \eqref{eq:levelset}. For each
\(\nu\in\mathbb{S}^n\), let \(s_0(\nu)\) denote the optimal admissible value for \((\Sigma_0,\Omega_0, E_0)\) in the direction \(\nu\). Then, 
\begin{equation*}
    s_0(\nu)\;\geq\;s_t(\nu)\qquad\text{for all }\nu\in\mathbb{S}^n \text{ and every } t\in[0,T),
\end{equation*}
where \(s_t(\nu)\) is the optimal admissible value for \((\Sigma_t, \Omega_t, E_t)\) in the direction \(\nu\).
\end{theorem}

\begin{proof}
Fix \(\nu\in\mathbb{S}^n\) and write \(s_0:=s_0(\nu)\). Let \(u_0\) be the signed distance function to \(\Sigma_0\). Since \(s_0\) is
admissible for the triple \((\Sigma_0,\Omega_0,E_0)\) with respect to \(\nu\), Lemma~\ref{lem: admissibility} implies that
\(s_0\) is admissible for \(u_0\) in the direction \(\nu\); that is,
\[
u_0(x)\geq u_0(x^*)\qquad\text{for all }x\in H_+(\nu,s_0),
\]
with equality along \(P_\nu(s_0)\), since \(x=x^*\) when \(x\in P_\nu(s_0)\).

Let \(u_t\) be the level-set solution of \eqref{eq:levelset} with initial data \(u_0\). Since the reflection $\mathcal R_{\nu,s_0}$ is an isometry of $\H^{n+1}$, the level-set equation~\eqref{eq:levelset} is invariant under $\mathcal R_{\nu,s_0}$; hence $u_t\circ\mathcal R_{\nu,s_0}$ is also a viscosity solution of~\eqref{eq:levelset} with initial data $u_0\circ\mathcal R_{\nu,s_0}$. By the comparison principle for viscosity solutions (cf.~\cite[Theorem~4.1]{CGG91} and~\cite[Theorem~3.2]{EvansSpruck91}), the inequality is preserved forward in time, that is
\[
u_t(x)\geq u_t(x^*)\qquad\text{for all }x\in H_+(\nu,s_0),\ t>0\,.
\]
 Hence \(s_0\) is admissible for \(u_t\) with respect to \(\nu\) for every \(t>0\). Applying
Lemma~\ref{lem: admissibility} again, we conclude that \(s_0\) is admissible for the triple \((\Sigma_t,\Omega_t,E_t)\) with
respect to \(\nu\) for all \(t>0\). By definition of the optimal admissible value,
\[
s_t(\nu)\leq s_0(\nu)
\]
for all \(t\in[0,T)\), as claimed.
\end{proof}

We aim to prove a uniform (in $t$) Lipschitz estimate for the evolving hypersurfaces (see Theorem~\ref{thm: Lipschitz bound}). To this end, we first record a monotonicity property of the signed distance function along normal geodesics issuing from an optimally placed hyperplane.

\begin{proposition}\label{prop: monotonicity}
Let \(\Sigma_0\) be the zero set of a continuous function $u_0$, and let \(\bar{s}\) be the overall optimal admissible value for $u$.
For any \(\nu\in\mathbb{S}^n\) and \(y\in P_\nu(\bar{s})\), let \(\gamma_y(r)\) be the geodesic (parametrized by arc-length) with initial
conditions
\[
\gamma_y(0)=y,\qquad \gamma_y'(0)=N(y)\,,
\]
where  \(N(y)\) is the unit normal to \(P_\nu(\bar{s})\) at \(y\) pointing into \(H_+(\nu,\bar{s})\). 
Then the function
\[
f_y(r):=u_0(\gamma_y(r))\,
\]
is non-decreasing for all $ r\geq 0$.
\end{proposition}

\begin{proof}
Let $0\leq r<r'$ and set $x:=\gamma_y(r)$, $x':=\gamma_y(r')$. Let $\nu_y:=\dot\gamma_y(\tfrac{r+r'}{2})$ be the unit tangent to $\gamma_y$ at the midpoint of $[x,x']$, and let $P_{\nu_y}(\tilde{s})$ be the totally geodesic hyperplane perpendicular to $\gamma_{\nu_y}$ that bisects $[x,x']$, where $\tilde s$ is the corresponding foliation parameter (see Figure~\ref{fig:mono}). Then $x'\in H_+(\nu_y,\tilde{s})$, the reflection $\mathcal R_{\nu_y,\tilde s}$ across $P_{\nu_y}(\tilde s)$ satisfies $(x')^*=x$, and by construction the midpoint of $[x,x']$ lies on $\gamma_y$ at hyperbolic distance $(r+r')/2>0$ from $P_\nu(\bar{s})$, so $\tilde{s}>\bar{s}$. Hence $\tilde{s}$ is admissible for the signed distance function $u_0$ in the direction $\nu_y$, and Lemma~\ref{lem: admissibility} yields
\[
u_0(x) \;=\; u_0\bigl((x')^*\bigr) \;\leq\; u_0(x')\,,
\]
which finishes the proof.
\begin{figure}[!ht]
    \centering
\includegraphics[width=10cm]{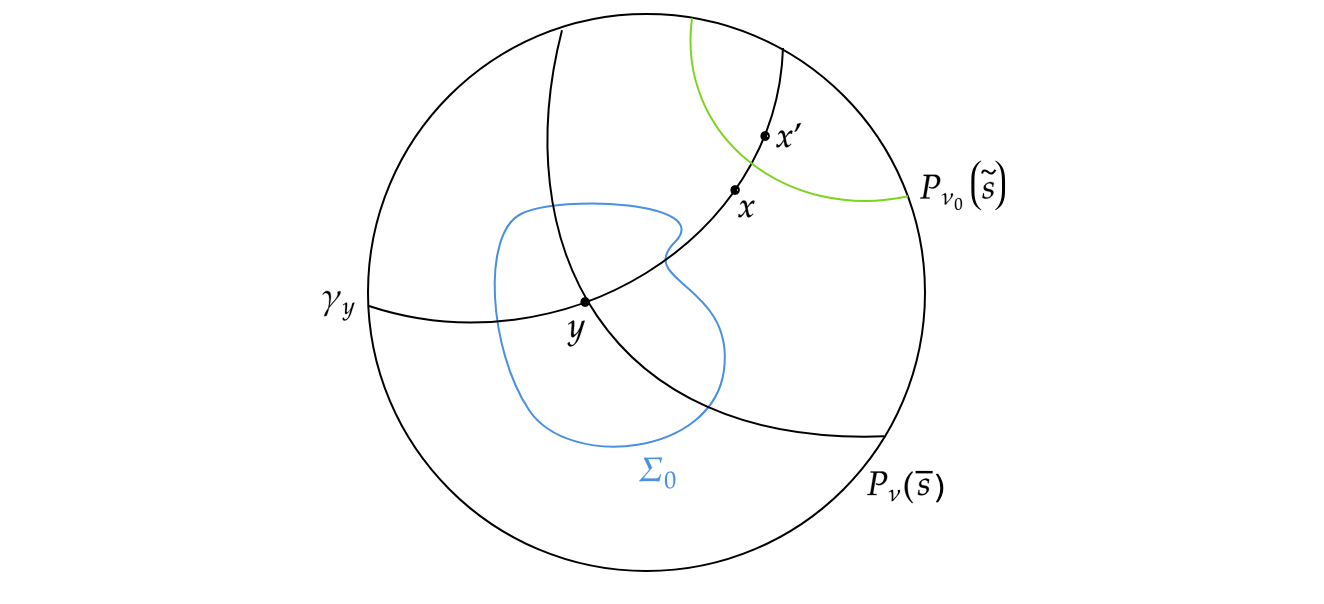}
    \caption{Monotonicity of \(f_y\) along the geodesic \(\gamma_y\).}
    \label{fig:mono}
\end{figure}
\end{proof}

We can now state the Lipschitz estimate for the evolving hypersurfaces in exponential coordinates.

\begin{theorem}\label{thm: Lipschitz bound}
Let $u$ be a solution to the level set flow given by \eqref{eq:levelset}. Let \(\bar{s}\) be the
overall optimal admissible value for \((\Sigma_0, \Omega_0, E_0)\). Then, for every \(\nu_0\in\mathbb{S}^n\) and every \(t\in(0,T)\), 
\[
\partial E_t\cap H_+(\nu_0,\bar{s})
\text{ and }
\partial \Omega_t\cap H_+(\nu_0,\bar{s})
\]
are Lipschitz graphs (in exponential coordinates) over \(P_{\nu_0}(\bar{s})\), with the Lipschitz bound independent of the time \(t\) and  the speed \(F\).
\end{theorem}

\begin{proof}
We begin by proving the first assertion, namely the graphicality of $\partial E_t$.
Let \(p_1,p_2\in\partial E_t\cap H_+(\nu_0,\bar{s})\). There exist \(y_i\in P_{\nu_0}(\bar{s})\) and \(d_i>0\) such that
\[
p_i=\exp_{y_i}(d_i\,N(y_i)),\qquad i=1,2,
\]
where \(N(y_i)\) is the unit normal to \(P_{\nu_0}(\bar{s})\) at $y_i$ pointing into \(H_+(\nu_0,\bar{s})\). Without loss of generality we assume \(0<d_1\le d_2\). 
Note that if $d_1=d_2$, then the result is trivial since the points $p_1$ and $p_2$ are contained in the same equidistant. Therefore we will assume that  $d_1<d_2$.

Let \(\mathcal E_{\nu_0}(d_1)\) denote the equidistant hypersurface at distance \(d_1\) from the totally geodesic hyperplane \(P_{\nu_0}(\bar{s})\) and passing through $p_1$.
We can then define
\[
\tilde{p}_2:=\exp_{y_2}(d_1\,N(y_2))\in \mathcal E_{\nu_0}(d_1).
\]

Observe that, by construction, the four points $\{y_1,y_2, p_1, \tilde p_2\}$ form a Saccheri quadrilateral of base $b := d_{\mathbb{H}^{n+1}}(y_1,y_2) $ and leg $l:= d_1$ (see Figure~\ref{fig:Equid}). Then, the summit $S := d_{\mathbb{H}^{n+1}}(p_1,\tilde{p}_2)$ can be computed explicitly as 
\begin{equation}\label{Saccheri}
\cosh d_{\mathbb{H}^{n+1}}(p_1,\tilde{p}_2) = \cosh d_{\mathbb{H}^{n+1}}(y_1,y_2) \cosh ^2 d_1 - \sinh ^2 d_1 .
\end{equation}

\begin{figure}[!ht]
    \centering

\tikzset{every picture/.style={line width=0.5pt}} 

\begin{tikzpicture}[x=0.65pt,y=0.65pt,yscale=-1,xscale=1]

\draw   (139,198) .. controls (139,94.17) and (223.17,10) .. (327,10) .. controls (430.83,10) and (515,94.17) .. (515,198) .. controls (515,301.83) and (430.83,386) .. (327,386) .. controls (223.17,386) and (139,301.83) .. (139,198) -- cycle ;
\draw    (139,198) -- (515,198) ;
\draw [color={rgb, 255:red, 208; green, 2; blue, 27 }  ,draw opacity=1 ]   (139,198) .. controls (192,118) and (460,112) .. (515,198) ;
\draw  [dash pattern={on 0.84pt off 2.51pt}]  (320,11) -- (321,198) ;
\draw    (379,199) .. controls (377,153) and (394,73) .. (443,50) ;
\draw    (320,135) .. controls (354,120) and (368,114) .. (410,78) ;
\draw [line width=0.75]    (321,198) -- (321,173) ;
\draw [shift={(321,171)}, rotate = 90] [color={rgb, 255:red, 0; green, 0; blue, 0 }  ][line width=0.75]    (10.93,-3.29) .. controls (6.95,-1.4) and (3.31,-0.3) .. (0,0) .. controls (3.31,0.3) and (6.95,1.4) .. (10.93,3.29)   ;
\draw   (372.24,198.18) -- (372,186) -- (378.76,186.82) ;
\draw    (410,78) -- (393.05,105.3) ;
\draw [shift={(392,107)}, rotate = 301.83] [color={rgb, 255:red, 0; green, 0; blue, 0 }  ][line width=0.75]    (10.93,-3.29) .. controls (6.95,-1.4) and (3.31,-0.3) .. (0,0) .. controls (3.31,0.3) and (6.95,1.4) .. (10.93,3.29)   ;
\draw  [dash pattern={on 0.84pt off 2.51pt}]  (156,123) -- (498,279) ;
\draw [line width=0.75]    (321,198) -- (333.13,172.8) ;
\draw [shift={(334,171)}, rotate = 115.71] [color={rgb, 255:red, 0; green, 0; blue, 0 }  ][line width=0.75]    (10.93,-3.29) .. controls (6.95,-1.4) and (3.31,-0.3) .. (0,0) .. controls (3.31,0.3) and (6.95,1.4) .. (10.93,3.29)   ;
\draw  [dash pattern={on 0.84pt off 2.51pt}]  (339,206) .. controls (357,151) and (364,133) .. (410,78) ;
\draw    (410,78) -- (384.44,102.61) ;
\draw [shift={(383,104)}, rotate = 316.08] [color={rgb, 255:red, 0; green, 0; blue, 0 }  ][line width=0.75]    (10.93,-3.29) .. controls (6.95,-1.4) and (3.31,-0.3) .. (0,0) .. controls (3.31,0.3) and (6.95,1.4) .. (10.93,3.29)   ;
\draw   (341.7,199.75) -- (352.96,202.85) -- (349,212) ;
\draw    (355,129) -- (354,142) ;
\draw    (417,139) -- (416,152) ;
\draw [color={rgb, 255:red, 74; green, 144; blue, 226 }  ,draw opacity=1 ][line width=1.5]    (355,137) .. controls (380,138) and (367,138) .. (416,144) ;

\draw (301,115) node [anchor=north west][inner sep=0.75pt]   [align=left] {$\displaystyle p_{1}$};
\draw (414,72) node [anchor=north west][inner sep=0.75pt]   [align=left] {$\displaystyle p_{2}$};
\draw (310,201) node [anchor=north west][inner sep=0.75pt]   [align=left] {$\displaystyle y_{1}$};
\draw (373,201) node [anchor=north west][inner sep=0.75pt]   [align=left] {$\displaystyle y_{2}$};
\draw (386,140) node [anchor=north west][inner sep=0.75pt]   [align=left] {$\displaystyle \tilde{p}_{2}$};
\draw (299,167) node [anchor=north west][inner sep=0.75pt]   [align=left] {$\displaystyle \nu _{0}$};
\draw (403,95.5) node [anchor=north west][inner sep=0.75pt]   [align=left] {$\displaystyle \xi ( 0)$};
\draw (336,174) node [anchor=north west][inner sep=0.75pt]   [align=left] {$\displaystyle \nu $};
\draw (368,74.5) node [anchor=north west][inner sep=0.75pt]   [align=left] {$\displaystyle \xi ( \nu )$};
\draw (429,128) node [anchor=north west][inner sep=0.75pt]   [align=left] {$\displaystyle \textcolor[rgb]{0.29,0.56,0.89}{V}$};

\end{tikzpicture}
    \caption{Comparison on the equidistant hypersurface \(\mathcal E_{\nu_0}(d_1)\). }
    \label{fig:Equid}
\end{figure}

Since $p_1, p_2\in H_+(\nu_0,\bar{s})$, which is open, there exists \(\epsilon>0\) (depending on $p_1, p_2$) such that, for all \(\nu\) in a small
neighborhood \(B_0(\nu_0,\epsilon)\subset\mathbb{S}^n\), we have
\[
p_i\in H_+(\nu,\bar{s}), \,
V(\tilde{p}_2,\epsilon):=
B(\tilde{p}_2,\epsilon)\cap \mathcal E_{\nu_0}(d_1)\subset H_+(\nu,\bar{s}).
\]

Let $[p_1, \tilde p _2]_{\mathcal E (d_1)}$ denote the (closed) {\it equidistant arc} joining $p_1$ and $\tilde p _2$ defined by 
$$[p_1, \tilde p _2]_{\mathcal E (d_1)} := \{ \exp_{y}(d_1\,N(y))\in \mathcal E_{\nu_0}(d_1) \, : \quad y \in [y_1 ,y_2]\}$$
For each \(\nu\in B_0(\nu_0,\epsilon)\), let \(\beta_\nu :[0, \infty) \to \mathbb{H}^{n+1}\) be the geodesic ray joining \(\beta_\nu (0) :=p_2\) to its closest point $y(\nu)$ in \(P_\nu(\bar{s})\). This is well-defined since $d_1 < d_2$.
Define
\[
\sigma(\nu):=\beta_\nu\cap \mathcal E_{\nu_0}(d_1),\qquad
\xi(\nu):=\beta_\nu'(0)\in T_{p_2}\mathbb{H}^{n+1}.
\]
Let
\([)
\theta(\nu):=\angle\big(\xi(\nu),\xi(\nu_0)\big)
\). Since $\sigma$ is continuous, there exists \(\delta>0\) such that
\[
\theta(\nu)\leq\delta\quad\Longrightarrow\quad \sigma(\nu)\in V(\tilde{p}_2,\epsilon).
\]
Note that $\delta$ and $\varepsilon$ depend only on $p_1$ and $p_2$. In particular, if $p_1$ and $p_2$ range over a fixed compact subset of $H^+(\nu_0, \bar s)$, then $\delta$ and $\varepsilon$ depend only on that compact set.

Consider the geodesic triangle $(p_2,p_1,\tilde{p}_2)_{\mathbb{H}^{n+1}}$, including the degenerate one, that is, if $p_1$ coincides with $\tilde p_2$, and let $\alpha_{\mathbb{H}^{n+1}}$ denote the internal angle of the triangle $(p_2,p_1,\tilde{p}_2)_{\mathbb{H}^{n+1}}$ at the vertex $p_2$. Observe that in the degenerate case $\alpha_{\mathbb{H}^{n+1}}=0$. 

First, we will show that the degenerate case is not possible, since $\alpha_{\mathbb{H}^{n+1}}$ is bounded below by $\delta$:

\begin{quote}
    \textbf{{Claim A}:} $\alpha_{\mathbb{H}^{n+1}}\geq \delta.$
\end{quote}
\begin{proof}[Proof of Claim A]
Assume $\alpha_{\mathbb{H}^{n+1}}<\delta$, then there exist $\bar\nu\in B_0(\nu_0,\epsilon)$ so that $\alpha_{\mathbb{H}^{n+1}}=\theta(\bar\nu)$ and $\sigma(\bar\nu)=p_1\in V(\tilde{p}_2,\epsilon),$ or equivalently there exist $\bar\nu\in B_0(\nu_0,\epsilon)$ and $\bar y\in P_\nu(\bar{s})$ so that 
\begin{equation*}
    p_1:=\exp_{\bar y}(r_1N(\bar y))\hspace{1 cm}\text{and}\hspace{1 cm}p_2=\exp_{\bar y}(r_2N(\bar y))
\end{equation*}
for some $r_1<r_2$.
 This follows from the fact that $\beta _{\bar \nu}$ is orthogonal to $\mathcal E_{\nu_0}(d_1)$ since $\bar y \in \mathcal E_{\nu_0}(d_1)$ is the   closest point to $p_2$.

Now, take a sequence $\{p_1^k\}\in E_t$, so that $p_1^k\to p_1$, as $k\to+\infty$, and note that $u_t(p_1^k)>0$ for all $k$. Thus, for $k$ big enough, we could use the same argument, that is, there exists $\nu_k\in B_0(\nu_0,\tilde{\epsilon})$ and $y_k \in P_{\nu_k}(\bar{s})$ so that
\begin{equation*}
    p_1^k:=\exp_{y_k}(r_1^kN(y_k))\:\:\text{and}\:\: p_2:=\exp_{y_k}(r_2^kN(y_k))
\end{equation*}
for some $r_1^k<r_2^k$. Finally, using Proposition \ref{prop: monotonicity}, we have that
\begin{equation*}
    0<u_t(p_1^k)\leq u_t(p_2)=0,
\end{equation*}
since $p_2\in \partial E_t$, which implies $u_t(p_2)=0$. This is a contradiction, and therefore Claim A holds.
\end{proof}

Next, we show that:
\begin{quote}
    \textbf{{Claim B}:} $|d_2-d_1|\leq \cot (\alpha_{\mathbb{H}^{n+1}}) G(y_1,y_2,d_1)$ where
    $$G(y_1,y_2,d_1) := \arccosh \left( \cosh d_{\mathbb{H}^{n+1}}(y_1,y_2) \cosh ^2 d_1 - \sinh ^2 d_1  \right).$$
\end{quote}
\begin{proof}[Proof of Claim B]
Let $(a,b,c)_{\mathbb{R}^{n+1}}$ be the Euclidean comparison triangle with the same side lengths as the geodesic triangle $(p_2,p_1,\tilde{p}_2)_{\mathbb{H}^{n+1}}$, with $a,b,c$ corresponding to $p_2,p_1,\tilde p_2$ respectively. Let $\alpha_{\mathbb R^{n+1}}, \gamma_{\mathbb R^{n+1}}, \beta_{\mathbb R^{n+1}}$ denote the Euclidean angles at $a, b, c$, and let $\alpha_{\mathbb H^{n+1}}, \beta_{\mathbb H^{n+1}}$ denote the hyperbolic angles at $p_2$ and $\tilde p_2$, respectively (so $\alpha_{\mathbb H^{n+1}}$ is the angle whose lower bound was established in Claim~A).

\smallskip

\noindent\emph{Step 1: $\beta_{\mathbb H^{n+1}}\ge\pi/2$.} The side $\tilde p_2 p_2$ is the prolongation of the leg $y_2\tilde p_2$ of the Saccheri quadrilateral $\{y_1,y_2,p_1,\tilde p_2\}$, because $p_2=\exp_{y_2}(d_2\,N(y_2))$ and $\tilde p_2=\exp_{y_2}(d_1\,N(y_2))$ lie on the same normal geodesic to $P_{\nu_0}(\bar s)$ with $d_1<d_2$. Hence $\beta_{\mathbb H^{n+1}}$ is the supplement (at $\tilde p_2$) of the Saccheri summit angle $\theta_S:=\angle(\tilde p_2 y_2,\tilde p_2 p_1)$. By Gauss--Bonnet applied to the Saccheri quadrilateral in $\H^{n+1}$, the sum of its four interior angles is strictly less than $2\pi$: with two right angles at the base vertices $y_1,y_2$ and equal summit angles $\theta_S$ at $p_1$ and $\tilde p_2$, this gives $\pi+2\theta_S<2\pi$, hence $\theta_S<\pi/2$. Therefore $\beta_{\mathbb H^{n+1}}=\pi-\theta_S>\pi/2$.

\smallskip

\noindent\emph{Step 2: $\beta_{\mathbb R^{n+1}}\ge\pi/2$.} The Reshetnyak (CAT($0$)) angle-comparison theorem applies vertex-by-vertex: in a CAT$(0)$ space (in particular in the CAT$(-1)$ space $\mathbb H^{n+1}$), the angle at each vertex of a geodesic triangle is at most the angle at the corresponding vertex of the Euclidean comparison triangle (\cite[Proposition~II.1.7(4)]{BridsonHaefliger}). Applying this at vertex $c=\tilde p_2$ of the single comparison triangle $(a,b,c)_{\mathbb R^{n+1}}$ yields $\beta_{\mathbb R^{n+1}}\ge\beta_{\mathbb H^{n+1}}\ge\pi/2$. Applying it at vertex $a=p_2$ of the same triangle yields $\alpha_{\mathbb R^{n+1}}\ge\alpha_{\mathbb H^{n+1}}$. Both inequalities thus refer to the same Euclidean comparison triangle.

\smallskip

\noindent\emph{Step 3: Euclidean trigonometry.} Since $\beta_{\mathbb R^{n+1}}\ge\pi/2$, the remaining Euclidean angles satisfy $\alpha_{\mathbb R^{n+1}}+\gamma_{\mathbb R^{n+1}}\le\pi/2$, so $\gamma_{\mathbb R^{n+1}}\le\pi/2-\alpha_{\mathbb R^{n+1}}$ and hence $\sin\gamma_{\mathbb R^{n+1}}\le\cos\alpha_{\mathbb R^{n+1}}$. By the Euclidean law of sines,
\[
\frac{|\overline{bc}|}{|\overline{ac}|}=\frac{\sin\alpha_{\mathbb R^{n+1}}}{\sin\gamma_{\mathbb R^{n+1}}}\;\ge\;\frac{\sin\alpha_{\mathbb R^{n+1}}}{\cos\alpha_{\mathbb R^{n+1}}}=\tan\alpha_{\mathbb R^{n+1}}\;\ge\;\tan\alpha_{\mathbb H^{n+1}},
\]
where the last inequality uses Step~2 and the monotonicity of $\tan$ on $[0,\pi/2)$ together with $\alpha_{\mathbb H^{n+1}}<\pi/2$, which follows from $\alpha_{\mathbb H^{n+1}}<\pi-\beta_{\mathbb H^{n+1}}\le\pi/2$ (as the three angles of the hyperbolic triangle sum to at most $\pi$).

\smallskip

\noindent\emph{Step 4: Conclusion.} Identifying $|\overline{bc}|=d_{\mathbb H^{n+1}}(p_1,\tilde p_2)$ and $|\overline{ac}|=d_{\mathbb H^{n+1}}(p_2,\tilde p_2)=|d_2-d_1|$ (the side lengths of the Euclidean comparison triangle agree with the hyperbolic ones by construction), and using~\eqref{Saccheri} to identify $d_{\mathbb H^{n+1}}(p_1,\tilde p_2)=G(y_1,y_2,d_1)$, we conclude $|d_2-d_1|\le\cot\alpha_{\mathbb H^{n+1}}\cdot G(y_1,y_2,d_1)$.
\end{proof}

Hence, by Claims A and B, one concludes that  \(\partial E_t\cap H_+(\nu_0,\bar{s})\) is a graph over \(P_{\nu_0}(\bar{s})\) with locally bounded slope. 
Since the argument is
purely geometric and only uses admissibility (which is preserved by Theorem~\ref{thm: admissible for all t}), the resulting
Lipschitz constant is locally uniform. In particular it  depends only on \(\Sigma_0\) and \(\bar{s}\), and not on \(t\) and \(F\).

The graphicality of \(\partial \Omega_t\cap H_+(\nu_0,\bar{s})\) is treated similarly, with the only difference that we should take instead an approximating sequence of points $p_2^k\in \Omega$ converging to $p_2$. 
\end{proof}

We will refine the Lipschitz estimate by exploiting the hyperbolic geometry more explicitly. In particular, one combines the description of \(H_+(\nu,\bar{s})\) in the Poincar\'e ball model with the radial parametrization of a hypersurface  in exponential coordinates and Gauss lemma to obtain a pointwise gradient estimate that depends only on \(\bar{s}\) and \(r_0\). This estimate will be used repeatedly in the sequel.

\begin{theorem}\label{thm: gradient bound}
Let $u$ be a solution to the level set flow given by \eqref{eq:levelset} and  let \(\bar{s}\) be the
overall optimal admissible value for \((\Sigma_0, \Omega_0, E_0)\). Then, the part of $\partial \Omega_t$ (and similarly $\partial E_t$) lying outside the geodesic ball $B({\bf 0},\bar{s})$ is a radial graph in exponential coordinates, that is,
    \begin{equation}
    \Phi_t(\nu):=\exp_0(r_t(\nu)\nu),\:\:\:\: \nu\in \S^n
    \end{equation}
    for a Lipschitz function $r_t : \mathbb{S}^n \to \mathbb{R}$, which, where differentiable, satisfies the gradient  estimate
    \begin{equation}\label{eq: grad estimate}
        |\nabla_{\mathbb{S}^{n}} r_t|\leq \frac{r_t\tanh\bar{s}}{\sqrt{\tanh^2(r_t)-\tanh^2(\bar{s})}}.
    \end{equation}
    
\end{theorem}
\begin{proof} We prove the statement for $\partial \Omega_t$; the proof for $\partial E_t$ is identical.
     Let $p_0\in \partial\Omega_t$ be a point so that $r_0:= d_{\mathbb{H}^{n+1}}(p_0,{\bf 0})>\bar{s}$. Let $\nu_0\in \S^n \subset T_0 \H^{n+1}$ be so that $\gamma_{\nu_0}(0)={\bf 0}$ and $\gamma_{\nu_0}(r_0)=p_0$.\\

\begin{quote}
   {\bf Claim A}: $p_0\in H^+(\nu,\bar{s})$ for any $\nu\in \S^n\subset T_0 \H^{n+1}$ satisfying $\langle \nu,\nu_0\rangle >\frac{\tanh(\bar{s})}{\tanh (r_0)}$. In particular, $p_0\in P_\nu(s)$ for some $s>\bar{s}$. 
\end{quote}   
\begin{proof}[Proof of the Claim A]
   By $\langle \nu,\nu_0\rangle >\frac{\tanh(\bar{s})}{\tanh (r_0)}$ and \eqref{rho} we have
   \begin{align*}
       \langle p_0,\nu\rangle&=\tanh\big(\frac{r_0}{2}\big)\langle \nu_0,\nu\rangle >\tanh(\frac{r_0}{2})\frac{\tanh(\bar{s})}{\tanh(r_0)}\\
       &=\frac{\tanh(\bar{s})\big(1+\tanh^2(\frac{r_0}{2})\big)}{2}=\frac{\tanh(\bar{s})(1+|p_0|^2)}{2}
   \end{align*}

Hence, \eqref{eq:Hplus-ball} implies $p_0\in H_+(\nu,\bar{s})$, which proves Claim A. 
\end{proof}

Now, since $\bar{s}\geq s_0(\nu_0)$, by Theorem~\ref{thm: Lipschitz bound} there exists a neighborhood $\mathcal{V}\subset \partial \Omega_t$ of $p_0$ which is a graph over a neighborhood $D \subset P_{\nu_0}(\bar{s})$. Hence, shrinking $\mathcal{V}$ if necessary, we can assume that $\mathcal{V}$ is a radial graph over a neighborhood $U(\nu_0) \subset \S^n$ in exponential coordinates, that is, 
    \begin{equation*}
        \mathcal{V}:=\{\exp_0(r_t(\nu)\nu):\nu\in U(\nu_0)\}
    \end{equation*}
    for some $r_t:U(\nu_0)\to\R$ such that $r_t(\nu_0)=r_0$. So we can parametrize $\Phi_t:U(\nu_0)\to \mathcal{V}$ where $\Phi_t(\nu):=\exp_0(r_t(\nu)\nu)$ and $\Phi_t(\nu_0)=p_0.$\\
    
Next, we obtain a gradient estimate for $r_t$, assuming that it is differentiable around $\nu_0$. Let $e\in T_{\nu_0}\S^n$, $\S^n\subset T_0\H^{n+1}$, be a unitary vector and $\nu:(-\epsilon,\epsilon)\to \S^n$ be a curve so that $\nu(0)=\nu_0$ and $\nu'(0)=e.$ Set 
\[
\tau:=\frac{d}{ds}_{|_{s=0}}\Phi_t(\nu(s))= d(\exp_0)_{r_0\nu_0}(x)\in T_{p_0}\partial \Omega_t,
\]
where $x:=d(r_t)_{\nu_0}(e)\nu_0+r_0e\in T_0\H^{n+1}$, see Figure \ref{fig:gradient}. By Claim A and Theorem \ref{thm: Lipschitz bound}, we can write $\partial \Omega_t$, locally around $p_0$, as a graph over $P_\nu(\bar{s})$ for any $\nu\in \S^n\subset T_0\H^{n+1}$ satisfying $\langle \nu,\nu_0\rangle > \frac{\tanh(\bar{s})}{\tanh(r_0)}$. Since $x/|x|$ is parallel to the tangent direction $\tau$, we must have
    \begin{equation*}
         \Big\langle\frac{x}{|x|},\nu_0\Big\rangle\leq \frac{\tanh(\bar{s})}{\tanh(r_0)}.
    \end{equation*}

 Now using Gauss lemma, we obtain
    \begin{align*}
      \Big\langle \frac{\tau}{|\tau|},\gamma_{\nu_0}'(r_0)\Big\rangle&=\Big\langle d(\exp_0)_{r_0\nu_0}\big(\frac{x}{|x|}\big),d(\exp_0)_{r_0\nu_0}(r_0\nu_0)\Big\rangle \\
       &=\Big \langle\frac{x}{|x|},\nu_0\Big\rangle =\frac{d(r_t)_{\nu_0}(e)}{\sqrt{|d(r_t)_{\nu_0}(e)|^2+r_0^2}}.
    \end{align*}
    Combining last two equations, we get the desired estimate.

\begin{figure}[!ht]
\centering

\tikzset{every picture/.style={line width=0.5pt}} 

\begin{tikzpicture}[x=0.5pt,y=0.5pt,yscale=-1,xscale=1]

\draw   (178.5,152.25) .. controls (178.5,75.34) and (240.84,13) .. (317.75,13) .. controls (394.66,13) and (457,75.34) .. (457,152.25) .. controls (457,229.16) and (394.66,291.5) .. (317.75,291.5) .. controls (240.84,291.5) and (178.5,229.16) .. (178.5,152.25) -- cycle ;
\draw    (412,51) -- (222.13,253.88) ;
\draw    (214,59) .. controls (263,119) and (370,118) .. (426,65) ;
\draw [color={rgb, 255:red, 74; green, 144; blue, 226 }  ,draw opacity=1 ]   (312,14) -- (314,156) ;
\draw  [color={rgb, 255:red, 208; green, 2; blue, 27 }  ,draw opacity=1 ] (262.41,156) .. controls (262.41,127.51) and (285.51,104.41) .. (314,104.41) .. controls (342.49,104.41) and (365.59,127.51) .. (365.59,156) .. controls (365.59,184.49) and (342.49,207.59) .. (314,207.59) .. controls (285.51,207.59) and (262.41,184.49) .. (262.41,156) -- cycle ;
\draw   (268,83) .. controls (291,93) and (262,79) .. (277,66) .. controls (292,53) and (337,95) .. (358,79) .. controls (379,63) and (360,129) .. (380,159) .. controls (400,189) and (257,189) .. (237,159) .. controls (217,129) and (245,73) .. (268,83) -- cycle ;
\draw [color={rgb, 255:red, 74; green, 144; blue, 226 }  ,draw opacity=1 ]   (292,14) .. controls (307,108) and (380,159) .. (456,142) ;
\draw [line width=1.5]    (314,104.41) -- (337,104.05) ;
\draw [shift={(340,104)}, rotate = 179.09] [color={rgb, 255:red, 0; green, 0; blue, 0 }  ][line width=1.5]    (14.21,-4.28) .. controls (9.04,-1.82) and (4.3,-0.39) .. (0,0) .. controls (4.3,0.39) and (9.04,1.82) .. (14.21,4.28)   ;
\draw    (313,72) -- (338.22,85.08) ;
\draw [shift={(340,86)}, rotate = 207.41] [color={rgb, 255:red, 0; green, 0; blue, 0 }  ][line width=0.75]    (10.93,-3.29) .. controls (6.95,-1.4) and (3.31,-0.3) .. (0,0) .. controls (3.31,0.3) and (6.95,1.4) .. (10.93,3.29)   ;

\draw (431,50) node [anchor=north west][inner sep=0.75pt]   [align=left] {$\displaystyle P_{\nu _{0}}(\overline{s})$};
\draw (310,161) node [anchor=north west][inner sep=0.75pt]   [align=left] {0};
\draw (314,52) node [anchor=north west][inner sep=0.75pt]   [align=left] {$\displaystyle p_{0}$};
\draw (263.33,26.33) node [anchor=north west][inner sep=0.75pt]  [color={rgb, 255:red, 74; green, 144; blue, 226 }  ,opacity=1 ] [align=left] {$\displaystyle \gamma _{\nu _{0}}$};
\draw (255,216) node [anchor=north west][inner sep=0.75pt]   [align=left] {$\displaystyle \gamma _{\nu }$};
\draw (187,104) node [anchor=north west][inner sep=0.75pt]   [align=left] {$\displaystyle \partial\Omega
 _{t}$};
\draw (323,212) node [anchor=north west][inner sep=0.75pt]   [align=left] {$\displaystyle \textcolor[rgb]{0.82,0.01,0.11}{B}\textcolor[rgb]{0.82,0.01,0.11}{_{0}}\textcolor[rgb]{0.82,0.01,0.11}{(}\textcolor[rgb]{0.82,0.01,0.11}{\overline{s}}\textcolor[rgb]{0.82,0.01,0.11}{)}$};
\draw (406,145) node [anchor=north west][inner sep=0.75pt]   [align=left] {$\displaystyle \textcolor[rgb]{0.29,0.56,0.89}{P}\textcolor[rgb]{0.29,0.56,0.89}{_{\nu }}\textcolor[rgb]{0.29,0.56,0.89}{(}\textcolor[rgb]{0.29,0.56,0.89}{\overline{s}}\textcolor[rgb]{0.29,0.56,0.89}{)}$};
\draw (316,107.41) node [anchor=north west][inner sep=0.75pt]  [font=\scriptsize] [align=left] {$\displaystyle e$};
\draw (328.5,86) node [anchor=north west][inner sep=0.75pt]  [font=\scriptsize] [align=left] {$\displaystyle \tau $};

\end{tikzpicture}
       \caption{Graphical representation of $\partial\Omega_t$ over a geodesic sphere with controlled gradient.}
       \label{fig:gradient}
\end{figure}

\end{proof}

In summary, the combination of the level-set formulation \eqref{eq:levelset} with the Aleksandrov reflection framework yields a robust package of a priori estimates for expanding curvature flows in \(\H^{n+1}\): admissibility is preserved along the flow, the evolving hypersurfaces become uniformly Lipschitz graphs in suitable coordinates.

\medskip

\subsection{Applications: Inverse Curvature Flows}\label{sec:icf}

In this section we specialize the general framework developed in Section~\ref{sec:levelset} to inverse curvature flows (ICF) of the form
\begin{equation}\label{ICF}
    \frac{\partial \phi _t}{\partial t} \;=\; -\,\frac{\eta _t}{\mathcal F},
\end{equation}
where $\eta_t$ is the \emph{inward} unit normal to $\Sigma_t$ as in Section~\ref{sec:levelset}; in the notation of~\eqref{eveq-again} this corresponds to the choice $F=-1/\mathcal F$, so that the inverse mean curvature flow ($\mathcal F=H$) is recovered as in~\eqref{eveq-again}. Here $\mathcal F$ is a curvature function of homogeneous degree $1$, monotone and concave, defined on a symmetric, convex, open cone $\Gamma\subset\R^{n+1}$, such that
\begin{equation}
    \mathcal F_{|\Gamma}>0 \qquad \text{and} \qquad \mathcal F_{|\partial\Gamma}=0.
\end{equation}
We also normalize \(\mathcal F\) by
\[
    \mathcal F(1,\dots,1)=n .
\]
These are the standard structural hypotheses in Gerhardt's theory of inverse curvature flows in hyperbolic space and include, in particular, the inverse mean curvature flow (IMCF) as the case \(\mathcal F=H\) \cite{Gerhardt2014}. Combined with the star-shapedness and gradient bounds, they allow us to upgrade viscosity solutions of \eqref{ICF} to smooth, strictly mean convex flows after a finite time. In this way we extend Harvie's eventual regularity result for weak IMCF in \(\H^{n+1}\) \cite{Harvie2024} from the speed \(\mathcal F=H\) to general inverse curvature speeds \(\mathcal{F}\), in a setting that is conceptually close to the Aleksandrov reflection framework of Chow--Gulliver in Euclidean space \cite{ChowGulliver01}.

We first look at the flow of a geodesic sphere \(\mathcal S_{r_0}\). Fix a point \(p_0 \in \mathbb{H}^{n+1}\) and consider geodesic polar coordinates centered at \(p_0\), recall \eqref{eq:polar-metric}. Geodesic spheres \(\mathcal S_r\) with center at \(p_0\) are umbilic, and their induced metric on \(\mathcal S_r\), $\bar{g}_{ij}$, and second fundamental form, $ \bar{h}_{ij}$, are given by
$$\bar{g}_{ij} \;=\; \sinh^2 r \,\sigma_{ij} \text{ and }\bar{h}_{ij} \;=\; \coth r\,\bar{g}_{ij},
$$respectively. Hence, if we consider an inverse curvature flow \eqref{ICF} with initial hypersurface \(\mathcal S_{r_0}\), then the evolving hypersurfaces \(\Sigma_t\) remain geodesic spheres with radii \(r(t)\) satisfying the scalar ODE
\begin{equation*}
    \dot{r}(t) \;=\; \frac{1}{\mathcal F(\kappa_1,\dots,\kappa_n)}
    \;=\; \frac{1}{\mathcal F(\coth r,\dots,\coth r)}
    \;=\; \frac{1}{n\coth r},
\end{equation*}
where in the last equality we used homogeneity and the normalization \(\mathcal F(1,\dots,1)=n\). Solving this ODE yields
\begin{equation}\label{sphere evolution}
    \sinh r(t) \;=\; e^{\frac{t}{n}}\sinh r_0\,.
\end{equation}

In the next theorem, we consider viscosity solutions to \eqref{ICF}  in the sense described in Section 3.

\begin{theorem}[Star-shapedness of ICF in hyperbolic space]\label{thm: Reg}
Let $u$ be the viscosity solution to \eqref{ICF} with initial condition $\Sigma_0$, and define
\begin{equation}\label{time}
    T=n\log\Bigg(\frac{\sinh(r_+)}{\sinh(r_-)}\Bigg),
\end{equation}
where $r_-$ and $r_+$ are the geodesic in-radius and out-radius of $\Omega_0$. Then for all $t\geq T$, $\partial\Omega_t$ and $\partial E_t$ become star-shaped and hence homeomorphic to $\mathbb{S}^n$.
\end{theorem}
  
\begin{proof}
   Note that the overall optimal  admissible value for $(\Sigma_0,\Omega_0, E_0)$, $\bar s$, satisfies $\bar s\le r^+$. Therefore, by Theorem~\ref{thm: gradient bound}, it suffices to show that 
    \[
        \partial \Omega_t \;\subset\; \mathbb{S}^{n}\times (r_+,\infty)
        \qquad \text{for all } t\in(T,\infty)\,.
    \]
    Indeed, by assumption we have
    \[
        \partial \Omega_0=\Sigma_0 \;\subset\; \mathbb{S}^n\times(r_-,\infty).
    \]
    Let \(\partial \Omega_0'=\Sigma_0'=\mathbb{S}^n\times\{r_-\}\) and consider the ICF \(\partial \Omega'_t=\Sigma_t'=\mathbb{S}^n\times\{r(t)\}\) with initial hypersurface \(\Sigma_0'\). As previously observed, its evolution satisfies
    \[
        \sinh(r(t)) \;=\; \sinh(r_-)\,e^{\frac{t}{n}}.
    \]
    By the avoidance principle for viscosity solutions of \eqref{ICF}, we have
    \[
        \mathbb{S}^n\times (0,r(t)) \;=\; \Omega_t' \;\subset\; \Omega_t
        \qquad \text{for all } t>0.
    \]
    In particular, for \(t=T\) we obtain
    \(
        \mathbb{S}^n\times (0,r_+) \;=\; \Omega_T' \;\subset\; \Omega_T,
    \)
    which implies the desired inclusion
    \[
        \Sigma_t \;\subset\; \mathbb{S}^{n}\times (r_+,\infty)
        \qquad \text{for all } t\in(T,\infty).
    \]
\end{proof}

If we further assume that \(\{\Sigma_t\}_{t\in(0, \infty)}\) is smooth, then by a result of Gerhardt \cite[Theorem~1.1]{Gerhardt2014} we conclude that the evolving hypersurfaces converge exponentially fast to a geodesic sphere at infinity. Hence, we obtain the following corollary, which extends Gerhardt's result by removing the star-shapedness assumption on the initial hypersurface.

\begin{corollary}
    Let \(\{\Sigma_t\}_{t\in(0, \infty)}\) be a smooth, compact solution to the inverse curvature flow \eqref{ICF}. Then, the hypersurfaces \(\Sigma_t\) escape to infinity, become strongly convex exponentially fast and approach totally umbilic geometry: the principal curvatures are uniformly bounded and converge exponentially fast to \(1\). In particular, the stability framework of Sahjwani--Scheuer~\cite{SahjwaniScheuer2024} applies to quantify the closeness of $\Sigma_t$ to a geodesic sphere via the quermassintegral inequalities at large $t$.
\end{corollary}

\begin{remark}
It is reasonable to conjecture that one does not need to assume a priori smoothness of the flow to apply Gerhardt's result. In other words, one may expect that once the flow becomes star-shaped, it also becomes smooth. This is indeed the case, for instance, for the inverse mean curvature flow in hyperbolic space of dimension \(3 \leq n \leq 7\), as shown by Li and Wei \cite{LiWei2017}; see also its application in \cite[Theorem 1.2]{Harvie2024}.
\end{remark}

Another corollary asserts that, if the initial hypersurface is not homeomorphic to a sphere, a singularity must occur in finite time:


\begin{corollary}
    Suppose that the compact initial hypersurface \(\Sigma_0\) is not homeomorphic to \(\mathbb{S}^n\). Then the corresponding viscosity solution to the level-set flow \eqref{eq:levelset} associated to \eqref{ICF} cannot remain a smooth hypersurface up to the time
    \[
        T \;=\; n\log\Big(\frac{\sinh(r_+)}{\sinh(r_-)}\Big),
    \]
    where \(r_-\) and \(r_+\) are the geodesic in-radius and out-radius of the initial domain \(\Omega_0\); equivalently, before time $T$, either a curvature singularity forms or the level-set flow exhibits a topological change (the boundary $\partial\Omega_t$ ceasing to be a smooth embedded hypersurface, including the possibility of fattening~\cite{EvansSpruck91,CGG91}).
\end{corollary}

\section{Hypersurfaces with one point on the asymptotic boundary}\label{section1pt}

In this section we work in the upper half-space model
\[
\mathbb{H}^{n+1}=\{(x,y)\in\R^n\times(0,\infty)\},
\]
and we keep the notation for the hyperbolic metric and the asymptotic boundary
introduced in the preliminaries. Recall that in  this model, totally geodesic hyperplanes are either vertical Euclidean hyperplanes, or
 Euclidean hemispheres whose boundary lies in \(\{y=0\}\) and which meet \(\{y=0\}\) orthogonally.
The corresponding reflections are given by Euclidean reflections across vertical hyperplanes and Euclidean inversions across hemispheres orthogonal to \(\{y=0\}\). 

\medskip

\subsection{Admissible foliations (non-compact case)}\label{AF-NC}
For each \(x\in\R^n\) we consider the vertical geodesic
\begin{equation*}
\gamma_x(s)=(x,s),\qquad s>0,
\end{equation*}
and the associated foliation of \(\mathbb{H}^{n+1}\) by totally geodesic
hyperplanes orthogonal to \(\gamma_x\). Concretely, for \(s>0\), we set
\[
P_x(s)=\Big\{p\in\mathbb{H}^{n+1} : |p-(x,0)|=s\Big\},
\]
which is a Euclidean half-sphere, centered at $(x,0)$, and  orthogonal to \(\{y=0\}\), and we denote by
\(\mathcal R_{x,s}\in\operatorname{Iso}(\mathbb{H}^{n+1})\) the reflection across the
hyperplane \(P_x(s)\). We also introduce the associated half-spaces
\begin{equation}\label{Hupper}
\begin{split}
    H_+(x,s)&=\big\{p\in\mathbb{H}^{n+1}: |p-(x,0) |<s\big\}, \\
    H_-(x,s)&=\big\{p\in\mathbb{H}^{n+1}: |p-(x,0) |>s\big\}.
\end{split}
\end{equation}
\medskip

\subsection{Optimal admissible values (non-compact case).}
Using the admissible foliations defined in Section~\ref{AF-NC}, we extend the notion of the optimal admissible value $\bar s$ (retaining the notation from the compact case) to properly embedded hypersurfaces with $\partial_\infty\Sigma=\{\infty\}$. Throughout this section, all hypersurfaces are assumed to be connected and properly embedded.  We will show that, under suitable expanding curvature flows, the evolving hypersurfaces become graphical over appropriate horospheres.

In the half-space model, horospheres are hyperplanes parallel to $\{y=0\}$, and we will denote them by
\[
\mathcal H(s):=\{y=s\}\,,\,\,s>0\,.
\]
We will also denote by $\mathcal H ^+ (s):=\{y>s\}$ and $\mathcal H^- (s):=\{y<s\}$ the mean-convex and  the mean-concave side of $\mathcal H (s)$, respectively.

\begin{definition}\label{Sigma-nc}
Let \(\Sigma_0\subset\mathbb{H}^{n+1}\) be a properly embedded \(C^2\) hypersurface with \(\partial_\infty\Sigma_0=\{\infty\}\). Let \(\Omega_0\) be the open
set with boundary \(\partial\Omega_0=\Sigma_0\) and asymptotic boundary \(\partial_\infty\Omega_0=\{\infty\}\). For
\(x\in\R^n\times\{0\}\), we say that \(s_0>0\) is \emph{admissible in the
direction \(x\)} if
\[
\Sigma_0\cap H_+(x,s) = \emptyset \text{ or } \mathcal R_{x,s}\big(\Sigma_0\cap H_+(x,s)\big)\subset\overline{\Omega}_0
\, \, \text{for all }s\in(0,s_0).
\]
The number
\[
s_0(x):=\sup\Big\{ s_0>0 \text{ is admissible in the direction }x\Big\}
\]
is called the \emph{optimal admissible value} for \(\Sigma_0\) in the direction
of \(x\in\R^n\times\{0\}\).

We denote the overall optimal admissible value by
\[
\bar{s}:=\inf\{s_0(x):x\in\R^n\times\{0\}\}\,.
\]
\end{definition}
Note that, in general, $\bar s$ can be zero.
Moreover, \(s_0(x)\) may be infinite for some \(x\). If \(s_0(x)=+\infty\) for
every \(x\in\R^n\times\{0\}\), then \(\Sigma_0\) is automatically a global graph
over any horosphere $\mathcal H(s)$.

We consider the curvature flow starting from the hypersurface \(\Sigma_0\) as in \eqref{eveq-again}. In particular, we study a one-parameter family 
\(\{\Sigma_t\}_{t\in[0,T)} \subset \mathbb{H}^{n+1}\) of properly embedded, connected, complete \(C^2\) hypersurfaces satisfying
\[
\partial_\infty \Sigma_t = \{\infty\} \qquad \text{for all } t \in [0,T),
\]
and evolving with normal speed \(F\) in the direction of the inward unit normal \(\eta_t\). Here, by inward we mean the unit normal pointing towards \(\Omega_t\), the open set with boundary \(\partial\Omega_t = \Sigma_t\) and asymptotic boundary \(\partial_\infty \Omega_t = \{\infty\}\).

In the compact setting, viscosity solutions of the flow are obtained via the level-set formulation introduced in Section~\ref{sec:levelset}, and the graphical description of the evolving hypersurfaces follows from the Aleksandrov reflection machinery developed there. For non-compact initial data with a single point on the asymptotic boundary, the existence, and especially the uniqueness, of viscosity solutions becomes more delicate. 

For this reason, we will \emph{assume} the existence of a viscosity solution with initial set \(\Sigma_0\); that is, a proper function \(u\) solving \eqref{eq:levelset} such that
\[
\partial_\infty(\partial E_t) = \{\infty\} = \partial_\infty(\partial \Omega_t).
\]
Under this assumption, the same reflection arguments yield graphical and containment properties analogous to those in the compact case. We collect this transfer in a lemma.

\begin{lemma}[Transfer of the reflection scheme to the one-point case]\label{lem:transfer-1pt}
Let $\Sigma\subset\H^{n+1}$ be a properly embedded $C^2$ hypersurface with $\partial_\infty\Sigma=\{\infty\}$, let $\Omega\subset\H^{n+1}\setminus\Sigma$ be the open set with $\partial\Omega=\Sigma$ and $\partial_\infty\Omega=\{\infty\}$, and let $E=\H^{n+1}\setminus\overline\Omega$. Then:
\begin{itemize}
   \item[(i)] $\partial_\infty\Sigma\cap\partial_\infty P_x(s)=\emptyset$ and $\Sigma\cap\overline{H_+(x,s)}$ is compact, for every $x\in\R^n\times\{0\}$ and every $s\in(0,+\infty)$. In particular, the Aleksandrov reflection method of Section~\ref{sec:Aleksandrov} can be applied with the foliation $\{P_x(s)\}_{s>0}$, and the first contact point (if any) occurs at a finite point of $\H^{n+1}$ (see~\cite{doCarmo} for details in the CMC case).
   \item[(ii)] With the admissibility notion of Definition~\ref{Sigma-nc}, Lemma~\ref{lem: admissibility}, Theorem~\ref{thm: admissible for all t}, and Proposition~\ref{prop: monotonicity} hold verbatim, with $P_\nu(s)$, $H_\pm(\nu,s)$, $\mathcal R_{\nu,s}$ replaced by $P_x(s)$, $H_\pm(x,s)$, $\mathcal R_{x,s}$.
\end{itemize}
\end{lemma}

The proof of (i) is immediate from the assumption $\partial_\infty\Sigma=\{\infty\}$ and the fact that $\overline{H_+(x,s)}$ is contained in a compact subset of $\overline{\H^{n+1}}$ not meeting~$\infty$. The proof of (ii) is a line-by-line reproduction of the compact-case arguments, using (i) to ensure that the reflection method is well-posed.

We remark that, due to non-compactness, this is entirely new material, even for smooth solutions. Therefore, in the remainder of this section we will focus on smooth solutions. However, we note that the results can be generalized to viscosity solutions as well, by replacing in the statements $\Sigma_t$, with the boundaries of $\{u<0\}$ and $\{u>0\}$.

\begin{theorem}\label{thmnc}
Let \(\Sigma_t\subset\mathbb{H}^{n+1}\), \(t\in(0,T)\), be a family of properly embedded \(C^2\) hypersurfaces with
\(\partial_\infty\Sigma_t=\{\infty\}\), such that \(\Sigma_0\) lies between two
horospheres \(\mathcal H(s_-)\) and \(\mathcal H(s_+)\) and evolves by \eqref{eveq-again}. Then, for
every \(x\in\R^n\times\{0\}\) and every \(t\in(0,T)\), the portion
\(
\Sigma_t\cap H_+(x,\bar{s})
\)
is a \(C^2\)-graph (in exponential coordinates) over \(P_{x}(\bar{s})\). Moreover, the gradient of this graph is locally bounded by a
constant that is independent of \(t\) and of the particular choice of the
curvature speed \(F\).
\end{theorem}

We now refine this description by working directly with the vertical projection
onto \(\R^n\times\{0\}\) and exploiting the explicit geometry of the
upper half-space model.

\begin{theorem}\label{thm:gradup}
For all \(t\in[0,T)\),  \(\Sigma_t\cap \mathcal H^-(\bar{s})\) is a graph over \(\R^n\times\{0\}\), that is, it is given as
\[
(x,r_t(x)),\qquad x\in U\subset \R^n,
\]
for a suitable function \(r_t:U\to(0,\bar{s})\), which additionally satisfies  the
gradient estimate
\begin{equation}\label{eq: grad-estimate-hor}
\big|\nabla _{\mathbb{R}^{n}} r_t(x)\big|\;\leq\;
\frac{\sqrt{\bar{s}^2-r_t(x)^2}}{r_t(x)}
\qquad\text{whenever }r_t(x)<\bar{s}.
\end{equation}
\end{theorem}

\begin{proof}
Fix \(t\in[0,T)\) and let \(p_0\in\Sigma_t\) be such that
\(
r_0:=|p_0- (x_0,0)|<\bar{s}
\),
where
\(x_0\in\R^n \equiv \R^n\times\{0\}\) is such that the vertical geodesic
\(\gamma_{x_0}\) satisfies
\[
\gamma_{x_0}(0)=(x_0,0),\qquad \gamma_{x_0}(r_0)=p_0.
\]

\begin{quote}
   {\bf Claim A}: For any \(x\in\R^n\times\{0\}\) with \(|x-x_0|<\sqrt{\bar{s}^2-r_0^2}\) one has \(p_0\in H_+(x,\bar{s})\).
\end{quote}

\begin{proof}[Proof of Claim A]
Indeed, if \(|x-x_0|<\sqrt{\bar{s}^2-r_0^2}\), then
\begin{align*}
  |p_0-(x,0)|^2
  &=|(x_0,r_0)-(x,0)|^2 =|x_0-x|^2+r_0^2\\
  &<\bar{s}^2-r_0^2+r_0^2=\bar{s}^2,
\end{align*}
so \(p_0\in H_+(x,\bar{s})\) by \eqref{Hupper}, as claimed (see Figure~\ref{fig:horograph}).
\end{proof}
\begin{figure}
    \centering

\tikzset{every picture/.style={line width=0.5pt}} 

\begin{tikzpicture}[x=0.65pt,y=0.65pt,yscale=-1,xscale=1]

\draw    (34,224) -- (628,223) ;
\draw    (302.4,7) -- (306.4,223) ;
\draw  [dash pattern={on 0.84pt off 2.51pt}]  (206,15) ;
\draw    (206,224.02) .. controls (238.6,111.22) and (354.2,95) .. (403,224) ;
\draw  [dash pattern={on 0.84pt off 2.51pt}]  (245,221) .. controls (287.8,93.4) and (423.8,106.2) .. (441,224) ;
\draw    (303.8,142.62) -- (360.6,224.6) ;
\draw [color={rgb, 255:red, 208; green, 2; blue, 27 }  ,draw opacity=1 ]   (41.4,157.02) .. controls (143.8,94.62) and (558.2,189.02) .. (620.6,162.62) ;
\draw [color={rgb, 255:red, 74; green, 144; blue, 226 }  ,draw opacity=1 ]   (36.6,121.02) -- (617.4,122.62) ;
\draw [color={rgb, 255:red, 74; green, 144; blue, 226 }  ,draw opacity=1 ]   (36.6,174.62) -- (625.4,177.82) ;
\draw [line width=1.5]    (303.8,142.62) -- (357.33,115.38) ;
\draw [shift={(360,114.02)}, rotate = 153.03] [color={rgb, 255:red, 0; green, 0; blue, 0 }  ][line width=1.5]    (14.21,-4.28) .. controls (9.04,-1.82) and (4.3,-0.39) .. (0,0) .. controls (4.3,0.39) and (9.04,1.82) .. (14.21,4.28)   ;
\draw [color={rgb, 255:red, 208; green, 2; blue, 27 }  ,draw opacity=1 ]   (63,148.02) -- (63,139.02) -- (63,130.02) ;
\draw [shift={(63,128.02)}, rotate = 90] [color={rgb, 255:red, 208; green, 2; blue, 27 }  ,draw opacity=1 ][line width=0.75]    (10.93,-3.29) .. controls (6.95,-1.4) and (3.31,-0.3) .. (0,0) .. controls (3.31,0.3) and (6.95,1.4) .. (10.93,3.29)   ;

\draw (302.8,222.4) node [anchor=north west][inner sep=0.75pt]   [align=left] {$\displaystyle x_{0}$};
\draw (282.4,138.8) node [anchor=north west][inner sep=0.75pt]   [align=left] {$\displaystyle p_{0}$};
\draw (613.2,139.8) node [anchor=north west][inner sep=0.75pt]  [color={rgb, 255:red, 208; green, 2; blue, 27 }  ,opacity=1 ] [align=left] {$\displaystyle \Sigma _{t}$};
\draw (353.2,221.8) node [anchor=north west][inner sep=0.75pt]   [align=left] {$\displaystyle \tilde{x}$};
\draw (161,194) node [anchor=north west][inner sep=0.75pt]   [align=left] {$\displaystyle P_{x_{0}}(\overline{s})$};
\draw (445,197) node [anchor=north west][inner sep=0.75pt]   [align=left] {$\displaystyle P_{\tilde{x}}(\overline{s})$};
\draw (35,128) node [anchor=north west][inner sep=0.75pt]   [align=left] {$\displaystyle \textcolor[rgb]{0.82,0.01,0.11}{\nu _{\Sigma _{t}}}$};
\draw (308.02,106.09) node [anchor=north west][inner sep=0.75pt]  [rotate=-340.05,xslant=-0.08] [align=left] {$\displaystyle \nabla _{\mathbb{R}^{n}} y( x_{0})$};
\end{tikzpicture}
    \caption{Gradient estimate for $\Sigma_t$ obtained via a totally geodesic comparison hypersurface}
    \label{fig:horograph}
\end{figure}

Since \(\bar{s}\le s_0(x_0)\), there exists a neighborhood
\(\mathcal{V}\subset\Sigma_t\) of \(p_0\) which is a graph over a neighborhood
\(D\subset P_{x_0}(\bar{s})\). Shrinking \(\mathcal{V}\) if necessary, we may
assume that \(\mathcal{V}\) is a graph over a neighborhood
\(U(x_0)\subset\R^n\times\{0\}\), namely
\[
\mathcal{V}=\big\{(x,r_t(x)):x\in U(x_0)\big\}
\]
for some \(C^2\) function \(r_t:U(x_0)\to(0,\bar{s})\) with
\(r_t(x_0)=r_0\).

Next we prove the gradient estimate. Let \(\tilde{x}\in\R^n\times\{0\}\) be such that
\(|\tilde{x}-x_0|=\sqrt{\bar{s}^2-r_0^2}\). Then the boundary of the upper
half-space \(P_{\tilde{x}}(\bar{s})\) can be written as a Euclidean graph over
\(\R^n\times\{0\}\) of the form
\[
y(x)=\sqrt{\bar{s}^2-|x-\tilde{x}|^2}.
\]
A direct computation gives
\[
\nabla _{\mathbb{R}^{n}}y(x)=\frac{x-\tilde{x}}{\sqrt{\bar{s}^2-|x-\tilde{x}|^2}}
=\frac{x-\tilde{x}}{y(x)},
\]
and hence
\[
|\nabla_{\mathbb{R}^{n}}y| (x_0)
=\frac{|x_0-\tilde{x}|}{\sqrt{\bar{s}^2-|x_0-\tilde{x}|^2}}
=\frac{\sqrt{\bar{s}^2-r_0^2}}{y(x_0)}
=\frac{\sqrt{\bar{s}^2-r_0^2}}{r_0}.
\]
Since a neighborhood of \(p_0\) in \(\Sigma_t\) lies inside \(H_+(\tilde{x},\bar{s})\), the graph of \(r_t\) over \(U(x_0)\) lies below the graph of \(y\) and is tangent to it at \(x_0\).
Therefore the usual comparison of gradients of tangent graphs yields
\[
|\nabla _{\mathbb{R}^{n}} r_t(x_0)|\le|\nabla_{\mathbb{R}^{n}} y(x_0)|
=\frac{\sqrt{\bar{s}^2-r_0^2}}{r_0},
\]
which is precisely \eqref{eq: grad-estimate-hor} at the point \(x_0\).
Since \(p_0\) was arbitrary in the region \(\{r_t<\bar{s}\}\), the estimate
holds on the whole set where \(\Sigma_t\) lies below \(\mathcal H(\bar{s})\).
\end{proof}

\medskip

\subsection{Inverse curvature flows}

The next result is the horospherical counterpart of Theorems~\ref{thm: Lipschitz bound} and~\ref{thm: gradient bound} in the non-compact setting: it shows that, when the asymptotic boundary reduces to a single point, the evolving hypersurfaces become global graphs over a fixed horosphere with uniform gradient control.

We keep the same structural assumptions on the curvature function $\mathcal F$ as in Section~3.4 (Gerhardt's setting \cite{Gerhardt2014}), which in particular include IMCF. 

We now show that if the initial hypersurface \(\Sigma_0\) lies between two horospheres, then the evolving hypersurfaces \(\Sigma_t\) remain between two explicitly controlled horospheres for all later times along the flow.

\begin{theorem}\label{thm:upbar}
Let $\{\Sigma_t\}_{t\in[0, \infty)}$ be a solution to the inverse curvature flow \eqref{ICF}. Then
\begin{itemize}
\item[(i)]
If \(\Sigma_0\) lies on the concave side of  the horosphere \(\mathcal H(s_+)\), $\Sigma_0\subset\mathcal H^-(s_+)$, then $\Sigma_t\subset  \mathcal H^-(s_+(t))$, where 
\[
s_+(t):=s_+e^{-\frac{t}{n}},\qquad t\ge0\,.
\]
\item[(ii)] If \(\Sigma_0\) lies on the convex side of  the horosphere \(\mathcal H(s_-)\), $\Sigma_0\subset\mathcal H^+(s_-)$, then $\Sigma_t\subset  \mathcal H^+(s_-(t))$, where 
\[
s_-(t):=s_- e^{-\frac tn}\left(1+\sqrt{1- e^{-\frac{2t}{n}}}\right)^{-1},\qquad t\ge0\,.
\]

\end{itemize}
\end{theorem}

\begin{proof}

\begin{figure}[!ht]
    \centering

\tikzset{every picture/.style={line width=0.5pt}} 

\begin{tikzpicture}[x=0.55pt,y=0.55pt,yscale=-1,xscale=1]

\draw    (29,263) -- (633,263) ;
\draw [color={rgb, 255:red, 208; green, 2; blue, 27 }  ,draw opacity=1 ]   (39,106) -- (615,105) ;
\draw    (43,174) -- (617,170) ;
\draw [color={rgb, 255:red, 74; green, 144; blue, 226 }  ,draw opacity=1 ]   (53,206) .. controls (191,102) and (565,282) .. (617,170) ;
\draw   (191,66) .. controls (191,43.91) and (208.91,26) .. (231,26) .. controls (253.09,26) and (271,43.91) .. (271,66) .. controls (271,88.09) and (253.09,106) .. (231,106) .. controls (208.91,106) and (191,88.09) .. (191,66) -- cycle ;
\draw    (331,0) -- (331,263) ;
\draw   (149,88) .. controls (149,41.61) and (186.61,4) .. (233,4) .. controls (279.39,4) and (317,41.61) .. (317,88) .. controls (317,134.39) and (279.39,172) .. (233,172) .. controls (186.61,172) and (149,134.39) .. (149,88) -- cycle ;
\draw [color={rgb, 255:red, 74; green, 144; blue, 226 }  ,draw opacity=1 ]   (233,172) -- (233,162) -- (233,154) ;
\draw [shift={(233,152)}, rotate = 90] [color={rgb, 255:red, 74; green, 144; blue, 226 }  ,draw opacity=1 ][line width=0.75]    (10.93,-3.29) .. controls (6.95,-1.4) and (3.31,-0.3) .. (0,0) .. controls (3.31,0.3) and (6.95,1.4) .. (10.93,3.29)   ;
\draw [color={rgb, 255:red, 208; green, 2; blue, 27 }  ,draw opacity=1 ]   (475,105.02) -- (475,95.02) -- (475,87.02) ;
\draw [shift={(475,85.02)}, rotate = 90] [color={rgb, 255:red, 208; green, 2; blue, 27 }  ,draw opacity=1 ][line width=0.75]    (10.93,-3.29) .. controls (6.95,-1.4) and (3.31,-0.3) .. (0,0) .. controls (3.31,0.3) and (6.95,1.4) .. (10.93,3.29)   ;
\draw [color={rgb, 255:red, 0; green, 0; blue, 0 }  ,draw opacity=1 ]   (231,106) -- (231,96) -- (231,88) ;
\draw [shift={(231,86)}, rotate = 90] [color={rgb, 255:red, 0; green, 0; blue, 0 }  ,draw opacity=1 ][line width=0.75]    (10.93,-3.29) .. controls (6.95,-1.4) and (3.31,-0.3) .. (0,0) .. controls (3.31,0.3) and (6.95,1.4) .. (10.93,3.29)   ;

\draw (619,173) node [anchor=north west][inner sep=0.75pt]   [align=left] {$\displaystyle \mathcal{H}(s_+(t))$};
\draw (615,88) node [anchor=north west][inner sep=0.75pt]  [color={rgb, 255:red, 208; green, 2; blue, 27 }  ,opacity=1 ] [align=left] {$\displaystyle \mathcal{H}( s_{+})$};
\draw (31,188) node [anchor=north west][inner sep=0.75pt]  [color={rgb, 255:red, 74; green, 144; blue, 226 }  ,opacity=1 ] [align=left] {$\displaystyle \Sigma _{t}$};
\draw (162,51) node [anchor=north west][inner sep=0.75pt]   [align=left] {$\displaystyle \mathcal{S}_{r_{0}}{}$};
\draw (223,265) node [anchor=north west][inner sep=0.75pt]   [align=left] {$\displaystyle x_{0}$};
\draw (224,62) node [anchor=north west][inner sep=0.75pt]   [align=left] {$\displaystyle \cdot $};
\draw (119,115) node [anchor=north west][inner sep=0.75pt]   [align=left] {$\displaystyle \mathcal{S}_{r(t)}$};
\draw (210,107) node [anchor=north west][inner sep=0.75pt]   [align=left] {$\displaystyle ( x_{0} ,s_{+})$};

\end{tikzpicture}

    \caption{Spheres $\mathcal{S}_{r_0}$ tangent to $\mathcal H(s_+)$ serve as a barriers, that force the flow to stay below a horosphere.}
    \label{fig:spherebarrier}
\end{figure}
Recall that any horosphere is umbilic with principal curvatures equal to $1$. Therefore, the evolution of any horosphere \(\mathcal H(s(0))=\{y=s(0)\}\) under the inverse curvature flow
\eqref{ICF}, is given by the family of horosphere $\{\mathcal H(s(t))\}_{t\in [0, \infty)}$, where 
\[
s(t)=s(0)e^{-\frac{t}{n}},\qquad t\ge0\,.
\] 
Since our solution is non-compact, evolving horospheres (which are themselves non-compact) cannot be used directly as barriers, as the avoidance principle does not apply to two non-compact solutions. Instead, we proceed as follows. As upper barriers, we use geodesic spheres tangent to the horosphere  \(\mathcal H(s_+)\) with arbitrarily large radius (see Figure~\ref{fig:spherebarrier}); as lower barriers, we use equidistant hypersurfaces tangent to the lower horosphere  \(\mathcal H(s_-)\) (see Figure~\ref{fig:equidistantbarrier}).

We first prove~(i). For each $x_0\in\R^n$ and each $r_0>0$, let $\mathcal S_{r_0}^{(x_0)}\subset\H^{n+1}$ denote the geodesic sphere of hyperbolic radius $r_0$ centered at the point $C(x_0):=(x_0,s_+ e^{r_0})$ of the vertical geodesic through $(x_0,s_+)$. Equivalently, $\mathcal S_{r_0}^{(x_0)}$ is the Euclidean sphere with center $(x_0,s_+ e^{r_0}\cosh r_0)$ and Euclidean radius $s_+ e^{r_0}\sinh r_0$; it is tangent to $\mathcal H(s_+)$ from above at $(x_0,s_+)$, and the rest of the sphere lies in $\mathcal H^+(s_+)=\{y>s_+\}$ (i.e.\ on the side of $\mathcal H(s_+)$ opposite to $\Sigma_0$).

Since \(\Sigma_0\subset \mathcal H^-(s_+)\) strictly and \(\mathcal S_{r_0}^{(x_0)}\subset \overline{\mathcal H^+(s_+)}\), the sphere \(\mathcal S_{r_0}^{(x_0)}\) is initially disjoint from \(\Sigma_0\). Indeed, this holds for every \(x_0\in\mathbb R^n\) and every \(r_0>0\), since the two sets lie in opposite closed half-spaces and the only point of \(S_{r_0}^{(x_0)}\) on the separating hyperplane \(\mathcal H(s_+)\) is \((x_0,s_+)\), which does not belong to \(\Sigma_0\).

By~\eqref{sphere evolution}, under~\eqref{ICF} the sphere $\mathcal S_{r_0}^{(x_0)}$ evolves to spheres $\{\mathcal S_{r(t)}^{(x_0)}\}_{t\ge 0}$ with the same hyperbolic center $C(x_0)$ and hyperbolic radius $r(t)$ satisfying
\begin{equation*}\label{r(t) evolution}
\sinh r(t)\;=\;e^{\frac{t}{n}}\sinh r_0\,.
\end{equation*}
By the avoidance principle for~\eqref{ICF} between the compact, smooth flow $\{\mathcal S_{r(t)}^{(x_0)}\}$ and the properly embedded flow $\{\Sigma_t\}$ (cf.~\cite[Section~2]{HuiskenIlmanen2001}; the principle applies whenever one of the two flows is compact), $\Sigma_t\cap\overline{B_{r(t)}^{(x_0)}}=\emptyset$ for every $t\ge 0$, where $B_{r(t)}^{(x_0)}$ denotes the open hyperbolic ball bounded by $\mathcal S_{r(t)}^{(x_0)}$.

\medskip

As $x_0$ varies over $\R^n$, the centers $C(x_0)$ trace out the horosphere $\mathcal H(s_+ e^{r_0})=\{y=s_+ e^{r_0}\}$. Using the identity $d_{\H^{n+1}}(p,\mathcal H(c))=|\log(y(p)/c)|$ in the upper half-space model, the union
\[
U(r_0,t)\;:=\;\bigcup_{x_0\in\R^n}B_{r(t)}^{(x_0)}\;=\;\bigl\{p\in\H^{n+1}:\;s_+ e^{r_0-r(t)}<y(p)<s_+ e^{r_0+r(t)}\bigr\}
\]
is the hyperbolic tube of radius $r(t)$ around $\mathcal H(s_+ e^{r_0})$, and $\Sigma_t\cap U(r_0,t)=\emptyset$ by the avoidance step applied for every $x_0$. The complement of $U(r_0,t)$ has two connected components, $\{y\le s_+ e^{r_0-r(t)}\}$ and $\{y\ge s_+ e^{r_0+r(t)}\}$. At $t=0$ one has $r(0)=r_0$, and $\Sigma_0\subset\{y<s_+\}=\{y<s_+ e^{r_0-r(0)+r_0}\}$ (note $s_+ e^{r_0-r(0)}=s_+$) lies in the lower component; by continuity in $t$ and connectedness of $\Sigma_t$, $\Sigma_t$ remains in the lower component for every $t\in[0,T_{\max})$, i.e.
\[
\Sigma_t\;\subset\;\bigl\{y\le s_+ e^{r_0-r(t)}\bigr\}\,.
\]
Finally, letting $r_0\to+\infty$ and using $\sinh r(t)=e^{t/n}\sinh r_0$, we get $r(t)=r_0+t/n+O(e^{-2r_0})$ and hence $s_+ e^{r_0-r(t)}\to s_+ e^{-t/n}$. Therefore
\[
\Sigma_t\;\subset\;\mathcal H^-(s_+(t))\,,\qquad s_+(t):=s_+ e^{-t/n}\,,
\]
as claimed.

\begin{figure}[!ht]
    \centering

\tikzset{every picture/.style={line width=0.5pt}} 

\begin{tikzpicture}[x=0.55pt,y=0.55pt,yscale=-1,xscale=1]

\draw    (40,261) -- (644,261) ;
\draw    (54,170) -- (628,166) ;
\draw [color={rgb, 255:red, 74; green, 144; blue, 226 }  ,draw opacity=1 ]   (62,139) .. controls (200,35) and (574,215) .. (626,103) ;
\draw    (342,-2) -- (342,261) ;
\draw  [draw opacity=0] (167.53,265.08) .. controls (167.49,263.99) and (167.46,262.88) .. (167.45,261.77) .. controls (166.9,213.17) and (196.9,173.43) .. (234.47,173.01) .. controls (272.04,172.58) and (302.94,211.63) .. (303.49,260.23) .. controls (303.5,261.16) and (303.5,262.08) .. (303.49,263.01) -- (235.47,261) -- cycle ; \draw   (167.53,265.08) .. controls (167.49,263.99) and (167.46,262.88) .. (167.45,261.77) .. controls (166.9,213.17) and (196.9,173.43) .. (234.47,173.01) .. controls (272.04,172.58) and (302.94,211.63) .. (303.49,260.23) .. controls (303.5,261.16) and (303.5,262.08) .. (303.49,263.01) ;  
\draw  [draw opacity=0] (168.71,264.29) .. controls (167.92,262.36) and (167.46,260.37) .. (167.37,258.35) .. controls (166.46,239.76) and (196.18,224.8) .. (233.73,224.94) .. controls (271.28,225.08) and (302.46,240.26) .. (303.36,258.84) .. controls (303.42,260.03) and (303.35,261.19) .. (303.17,262.34) -- (235.36,258.6) -- cycle ; \draw   (168.71,264.29) .. controls (167.92,262.36) and (167.46,260.37) .. (167.37,258.35) .. controls (166.46,239.76) and (196.18,224.8) .. (233.73,224.94) .. controls (271.28,225.08) and (302.46,240.26) .. (303.36,258.84) .. controls (303.42,260.03) and (303.35,261.19) .. (303.17,262.34) ;  
\draw  [dash pattern={on 0.84pt off 2.51pt}]  (49,225) -- (626,223) ;
\draw [color={rgb, 255:red, 208; green, 2; blue, 27 }  ,draw opacity=1 ] [dash pattern={on 4.5pt off 4.5pt}]  (146,0) -- (146,261) ;
\draw [color={rgb, 255:red, 208; green, 2; blue, 27 }  ,draw opacity=1 ] [dash pattern={on 4.5pt off 4.5pt}]  (317,1) -- (317,261) ;
\draw [color={rgb, 255:red, 0; green, 0; blue, 0 }  ,draw opacity=1 ]   (195,170.02) -- (195,160.02) -- (195,152.02) ;
\draw [shift={(195,150.02)}, rotate = 90] [color={rgb, 255:red, 0; green, 0; blue, 0 }  ,draw opacity=1 ][line width=0.75]    (10.93,-3.29) .. controls (6.95,-1.4) and (3.31,-0.3) .. (0,0) .. controls (3.31,0.3) and (6.95,1.4) .. (10.93,3.29)   ;
\draw [color={rgb, 255:red, 74; green, 144; blue, 226 }  ,draw opacity=1 ]   (355,121.02) -- (355,111.02) -- (355,103.02) ;
\draw [shift={(355,101.02)}, rotate = 90] [color={rgb, 255:red, 74; green, 144; blue, 226 }  ,draw opacity=1 ][line width=0.75]    (10.93,-3.29) .. controls (6.95,-1.4) and (3.31,-0.3) .. (0,0) .. controls (3.31,0.3) and (6.95,1.4) .. (10.93,3.29)   ;
\draw [color={rgb, 255:red, 0; green, 0; blue, 0 }  ,draw opacity=1 ]   (232,225.02) -- (232,215.02) -- (232,207.02) ;
\draw [shift={(232,205.02)}, rotate = 90] [color={rgb, 255:red, 0; green, 0; blue, 0 }  ,draw opacity=1 ][line width=0.75]    (10.93,-3.29) .. controls (6.95,-1.4) and (3.31,-0.3) .. (0,0) .. controls (3.31,0.3) and (6.95,1.4) .. (10.93,3.29)   ;

\draw (613,145) node [anchor=north west][inner sep=0.75pt]   [align=left] {$\displaystyle \mathcal{H}( s_{-})$};
\draw (41,129) node [anchor=north west][inner sep=0.75pt]  [color={rgb, 255:red, 74; green, 144; blue, 226 }  ,opacity=1 ] [align=left] {$\displaystyle \Sigma _{0}$};
\draw (234,263) node [anchor=north west][inner sep=0.75pt]   [align=left] {$\displaystyle x_{0}$};
\draw (232,255) node [anchor=north west][inner sep=0.75pt]   [align=left] {$\displaystyle \cdot $};
\draw (137,263) node [anchor=north west][inner sep=0.75pt]   [align=left] {$\displaystyle x_{1}$};
\draw (310,266) node [anchor=north west][inner sep=0.75pt]   [align=left] {$\displaystyle x_{2}$};
\draw (599,226) node [anchor=north west][inner sep=0.75pt]   [align=left] {$\displaystyle \mathcal{H}( s_{-}( t))$};
\draw (156,173) node [anchor=north west][inner sep=0.75pt]   [align=left] {$\displaystyle \mathcal{E}_{d_{0}}$};
\draw (225,184) node [anchor=north west][inner sep=0.75pt]   [align=left] {$\displaystyle \mathcal{E}_{d_{t}}$};
\draw (213,144) node [anchor=north west][inner sep=0.75pt]   [align=left] {$\displaystyle ( x_{0} ,s_{-})$};

\end{tikzpicture}
    \caption{Equidistant hypersurface $\mathcal E_{d_0}$ in $\mathcal H^+(s_-)$ serve as barriers that force the flow to stay above a horosphere.}
    \label{fig:equidistantbarrier}
\end{figure}

We next prove (ii). Fix a point \((x_0,s_-)\in \mathcal H(s_-)\) and consider an equidistant
hypersurface \(\mathcal E_{d_0}\) at (signed) distance \(d_0>0\) from a totally
geodesic hyperplane which is tangent to \(\mathcal H(s_-)\) at \((x_0,s_-)\), and contained in 
\(\overline {\mathcal H}^+(s_-)\), as in Figure~\ref{fig:equidistantbarrier}. The evolution of \(\mathcal{E}_{d_0}\) under the inverse
curvature flow \eqref{ICF}, is given by the family of equidistant
hypersurfaces \(\mathcal E_{d(t)}\), where 
\begin{equation}\label{d(t) evolution}
\cosh d(t)=e^{\frac{t}{n}}\cosh d_0\,.
\end{equation}
We can now apply the avoidance principle between the evolution of the equidistant hypersurfaces and our solution, which, as in case (i), yields $\Sigma_t\subset \mathcal H^+( s_-(t))$, where 
\[
s_-(t)= s_-e^{-d(t)}\,.
\] 
Note that even though both are evolutions of non compact surfaces the avoidance principle is still applicable as the ends of the two solutions are far apart. To make this more clear, one can consider for example restricting to a vertical strip
\[V(x_1,x_2)=\{x_1<x<x_2\}\times(0,\infty)\] containing \(\mathcal E_{d_0}\), so that
\(\Sigma_t':=\Sigma_t\cap V(x_1,x_2)\) is compact. Since 
\(\Sigma_0'\cap \mathcal E_{d_0}=\emptyset\) and $\partial \Sigma_t=\Sigma_t\cap \partial V(x_1, x_2)$ which is disjoint from $\mathcal E_{(d(t))}$, the avoidance principle yields
that \(\Sigma_t'\) stays disjoint from \(\mathcal E_{d(t)}\) and hence so does $\Sigma_t$.

Since the above is true for any $d_0>0$, the result follows after observing that  
\[
e^{-d(t)}\stackrel{d_0\to 0}{\longrightarrow} e^{-\frac tn}\left(1+\sqrt{1- e^{-\frac{2t}{n}}} \right)^{-1}\ge \frac12 e^{-\frac{t}{n}}\,.
\]

\end{proof}

Combining (i) of Theorem~\ref{thm:upbar} with Theorem~\ref{thm:gradup}, we see
that for all times 
\[
t\ge T:= n\log\frac{s_+}{\bar s}
\]
the hypersurfaces \(\Sigma_t\) are global graphs over the
horosphere \(\mathcal H(\bar{s})\).

Moreover, combining Theorem~\ref{thm:upbar} with Theorem~\ref{thm:gradup} we see that, for large times, $\Sigma_t$ remains trapped between two parallel horospheres and is a global graph over $\mathcal H(\bar{s})$ with uniform Euclidean-gradient bound. Specialising to the inverse mean curvature flow, these properties match the hypotheses of \cite[Theorem~1]{Allen}; we record the resulting statement below.

\begin{theorem}[Non-compact IMCF with one point at infinity]\label{thm:NonCompact}
Let $\Sigma_0\subset\H^{n+1}$ be a smooth, properly embedded hypersurface with $\partial_\infty\Sigma_0=\{\infty\}$ lying between two horospheres $\mathcal H(s_-)$ and $\mathcal H(s_+)$ ($0<s_-<s_+<\infty$), with $0<H_0\le H_{\Sigma_0}\le H_1<\infty$ and $|A_{\Sigma_0}|\le A_0<\infty$. Let $T_{\max}\in(0,+\infty]$ denote the maximal time of classical existence of the inverse mean curvature flow starting at $\Sigma_0$. Then:
\begin{itemize}
   \item[(a)] (\emph{a priori bounds}) By Theorem~\ref{thm:upbar}, the flow remains between the horospheres $\mathcal H(s_-(t))$ and $\mathcal H(s_+(t))$, with $s_\pm(t)$ as in Theorem~\ref{thm:upbar}; in particular $\Sigma_t$ stays in a fixed slab of $\H^{n+1}$ on any bounded time interval.
   \item[(b)] (\emph{graphicality and gradient bound}) For $t\ge T := n\log(s_+/\bar s)$, Theorem~\ref{thm:gradup} represents $\Sigma_t$ as a global graph over $\R^n\times\{0\}$ with Euclidean gradient bounded by $\sqrt{\bar s^{\,2}-r_t^{\,2}}/r_t$.
   \item[(c)] (\emph{long-time existence and convergence}) $T_{\max}=+\infty$, and $\Sigma_t$ converges asymptotically to a horosphere as $t\to+\infty$.
\end{itemize}
\end{theorem}

\begin{proof}[Proof of (c)]
We verify the hypotheses of~\cite[Theorem~1]{Allen} for $\Sigma_T$, viewed as the initial datum of an IMCF starting at time~$T$.~\cite[Theorem~1]{Allen} requires that the initial hypersurface satisfy:
\begin{enumerate}
\item[(A1)] bounded mean curvature: $0<H_0\le H(\cdot,0)\le H_1<\infty$,
\item[(A2)] bounded second fundamental form: $|A|(\cdot,0)\le A_0<\infty$,
\item[(A3)] global graphicality over $\R^n\times\{0\}$ in the upper half-space model, with bounded height and bounded Euclidean gradient,
\item[(A4)] uniform separation from $\R^n\times\{0\}$: the graph function $y$ satisfies $y\ge c>0$ for some constant $c$.
\end{enumerate}
Hypothesis (A3) is satisfied at time $T$ by (b): the graph function $r_T$ is bounded above by $s_+(T)<\infty$ (from (a)) and has Euclidean gradient bounded by $\sqrt{\bar s^{\,2}-r_T^{\,2}}/r_T\le\sqrt{\bar s^{\,2}-s_-(T)^{\,2}}/s_-(T)$ (from~\eqref{eq: grad-estimate-hor} and (a)). Hypothesis (A4) is satisfied at time $T$ by (a): $r_T\ge s_-(T)>0$. Hypotheses (A1) and (A2) hold at time $0$ by assumption ($0<H_0\le H_{\Sigma_0}\le H_1$ and $|A_{\Sigma_0}|\le A_0$); these bounds are propagated to time $T$ along the smooth IMCF on the bounded interval $[0,T]\subset[0,T_{\max})$ by standard parabolic regularity for fully nonlinear curvature flows (interior Schauder estimates combined with the evolution equations for $H$ and $|A|^2$; see, e.g., \cite[Section~3]{Gerhardt2014} for the smooth-flow estimates in $\H^{n+1}$, which apply on any bounded time interval on which the flow is classical). In particular, $\Sigma_T$ satisfies (A1) and (A2) with constants $H_0(T),H_1(T),A_0(T)$ depending on $T$ and the initial bounds. Applying~\cite[Theorem~1]{Allen} to $\Sigma_T$, viewed as a new initial datum, yields long-time existence beyond $T$ and asymptotic convergence of $\Sigma_t$ to a horosphere as $t\to+\infty$. Since the flow exists smoothly on $[0,T]$ by hypothesis, we conclude $T_{\max}=+\infty$ overall.
\end{proof}

\section{Hypersurfaces with two points on the asymptotic boundary}

In this section we continue working in the upper half-space model
\[
\mathbb{H}^{n+1}=\{(x,y)\in\R^n\times(0,\infty)\},
\]
as we did in Section \ref{section1pt}. Since we will be considering non compact surfaces whose asymptotic boundary consists of two points, we fix
\[
\{{\bf 0},\infty\}\subset \partial_\infty\H^{n+1}
\]
 and let $\gamma$ be the vertical geodesic with (asymptotic) endpoints ${\bf 0}$ and $\infty$, that is, 
 \[
\gamma(s)=(0,s),\qquad s>0\,.
\]

\subsection{Admissible foliations}\label{AF-2pts}
For each $x\in\partial_\infty\H^{n+1}\setminus\{{\bf 0},\infty\}$, we denote by  $\gamma_x$ the  unique geodesic perpendicular to $\gamma$ with $\gamma_x(0)\in\gamma$ and $\gamma_x(+\infty):=\lim_{s\to+\infty}\gamma_x(s)=x$. We also consider the associated foliation of the half-space
$$H_+(x):=\{p\in\H^{n+1}:\la p,x\ra>0\}$$
by totally geodesic hyperplanes $\{P_x(s)\}_{s\in(0,+\infty)}$ orthogonal to $\gamma_x$, as in see Figure~\ref{fig:fol2pts}. We denote by $\mathcal R_{x,s}\in\mathrm{Iso}(\H^{n+1})$ the reflection across $P_x(s)$. We also introduce the associated half-spaces
\begin{equation}\label{Hupper2pts}
\begin{split}
    H_+(x,s)&=\bigcup_{\tilde{s}\in(s,+\infty)}P_x(\tilde{s}), \\
    H_-(x,s)&=\bigcup_{\tilde{s}\in(0,s)}P_x(\tilde{s}).
\end{split}
\end{equation}

\begin{figure}[!ht]
    \centering
    \includegraphics[width=0.8\linewidth]{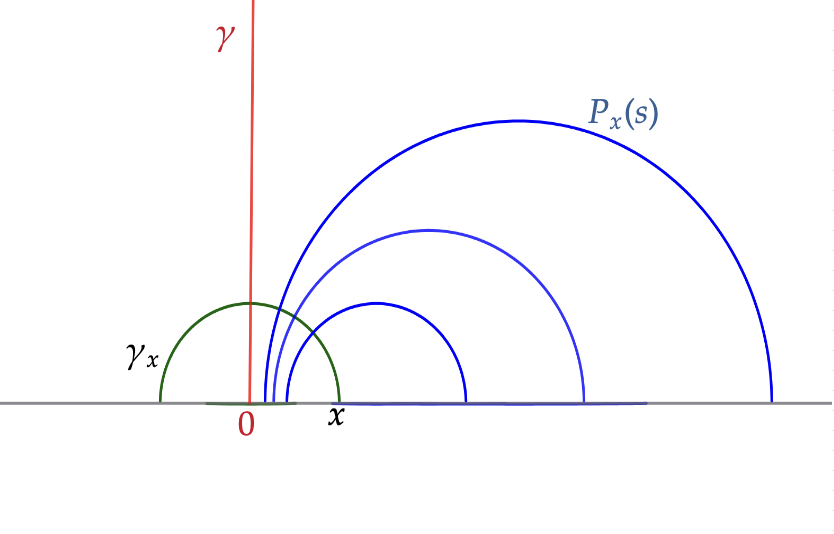}
    \caption{Admissible foliation}
    \label{fig:fol2pts}
\end{figure}
\subsection{Optimal admissible values.}
Using the admissible foliations defined in subsection \ref{AF-2pts}, we define the notion of the optimal admissible value $\bar s$ to properly embedded hypersurfaces with \(\partial_\infty\Sigma=\{{\bf 0},\infty\}\). We will show that, under suitable expanding curvature
flows, the evolving hypersurfaces become graphical over appropriate
cylinders whose axis is $\gamma$. In the halfspace model, these cylinders are given by
\begin{equation}\label{cylinder}
    \mathcal{C}(\rho):=\{p\in\H^{n+1}:d_{\H^{n+1}}(\gamma,p)=\rho\}\,,\quad\rho>0\,,
\end{equation}
which, setting $\lambda_\rho:=1/\sinh\rho$, is the Euclidean cone
\begin{equation*}
    \mathcal{C}(\rho)\;=\;\{(x,y)\in\R^n\times(0,\infty):\;y=\lambda_\rho\,|x|\}\,.
\end{equation*}
We denote by $\mathcal C^-(\rho):=\{p\in\H^{n+1}:d_{\H^{n+1}}(\gamma,p)<\rho\}$ and $\mathcal C^+(\rho):=\{p\in\H^{n+1}:d_{\H^{n+1}}(\gamma,p)>\rho\}$ the mean-convex (containing $\gamma$) and the mean-concave side of $\mathcal C(\rho)$, respectively. Recall that $\mathcal C(\rho)$ is isoparametric with two distinct principal curvatures: $\coth\rho$, with multiplicity $n-1$ (cross-section directions, with respect to the inward unit normal pointing towards $\gamma$), and $\tanh\rho$, with multiplicity $1$ (axial direction along $\gamma$). Its mean curvature is $H_{\mathcal C(\rho)}=(n-1)\coth\rho+\tanh\rho$, which satisfies $H_{\mathcal C(\rho)}>n$ for every $\rho<\infty$ and decreases monotonically to $n$ as $\rho\to\infty$ (see, e.g.,~\cite{CecilRyan}).

\begin{definition}\label{admis tp}
Let \(\Sigma_0\subset\mathbb{H}^{n+1}\) be a properly embedded \(C^2\) hypersurface with \(\partial_\infty\Sigma_0=\{{\bf 0},\infty\}\). Let \(\Omega_0 \subset \mathbb{H}^{n+1} \setminus \Sigma _0\) be the open
set with boundary \(\partial\Omega_0=\Sigma_0\) and asymptotic boundary \(\partial_\infty\Omega_0=\{{\bf 0},\infty\}\). For
$x\in\partial_{\infty}\H^{n+1}\setminus\{{\bf 0},\infty\}$, we say that \(s_0>0\) is \emph{admissible w.r.t. $x$} if 
\[
\Sigma_0\cap H_+(x,s) = \emptyset \text{ or } \mathcal R_{x,s}\big(\Sigma_0\cap H_+(x,s)\big)\subset\overline{\Omega}_0
\, \, \text{for all }s\in(s_0,\infty).
\]
\begin{figure}[!ht]
    \centering
    \includegraphics[width=0.8\linewidth]{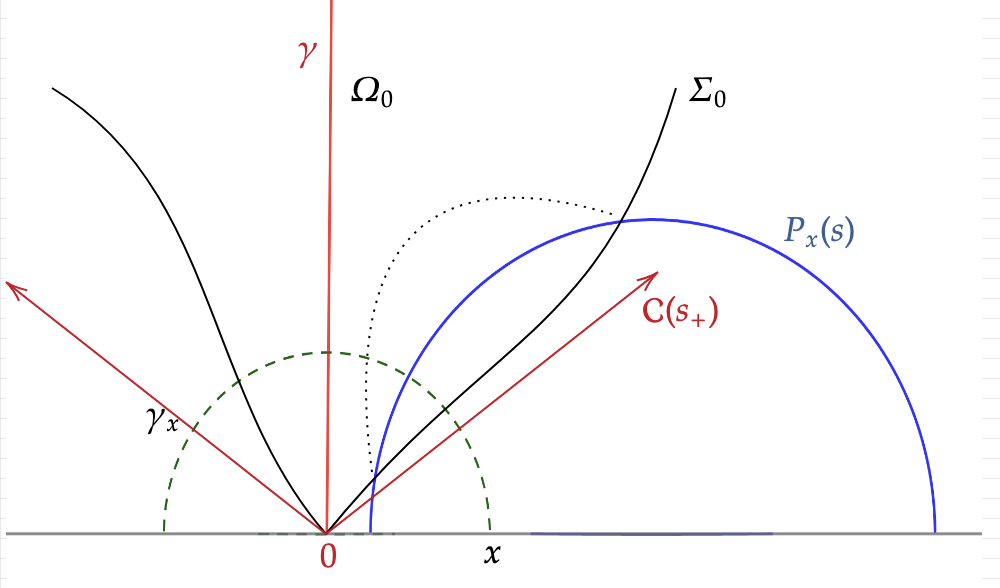}
    \caption{Admissibility of $\Sigma_0$}
    \label{fig:adm2pts}
\end{figure}

\noindent The number
\[
s_0(x)\;:=\;\inf\bigl\{\,s>0\;:\;s\text{ is admissible w.r.t.\ }x\,\bigr\}\;\in\;(0,+\infty)
\]
is called the \emph{optimal admissible value} for $\Sigma_0$ with respect to $x\in\partial_{\infty}\H^{n+1}\setminus\{{\bf 0},\infty\}$, with the convention $s_0(x)=+\infty$ when no $s>0$ is admissible.\\
The \emph{overall optimal admissible value} is
\[
\bar{s}\;:=\;\sup\bigl\{\,s_0(x)\;:\;x\in\partial_{\infty}\H^{n+1}\setminus\{{\bf 0},\infty\}\,\bigr\}\;\in\;(0,+\infty]\,.
\]
\end{definition}
\noindent
Two extreme behaviours will be useful:
\begin{itemize}
\item[(a)] If $s_0(x)=0$ for every $x\in\partial_{\infty}\H^{n+1}\setminus\{{\bf 0},\infty\}$, then $\Sigma_0$ is a $C^2$-graph in exponential coordinates over its orthogonal projection $\pi_{\mathcal C(s)}(\Sigma_0)\subset\mathcal C(s)$, for every $s>0$. To see this, fix $p_0\in\Sigma_0$ and $s>0$, and let $y_*\in\mathcal C(s)$ be the unique closest point to $p_0$ in $\mathcal C(s)$ (well-defined since $d_{\H^{n+1}}(\gamma,\cdot)$ is smooth and strictly convex on $\H^{n+1}\setminus\gamma$). Choose $x\in\partial_\infty\H^{n+1}\setminus\{\mathbf 0,\infty\}$ so that the totally geodesic hemisphere $P_x(s)$ is tangent to $\mathcal C(s)$ exactly at $y_*$. Since $s>s_0(x)=0$, $s$ is admissible w.r.t.\ $x$, so $\mathcal R_{x,s}(\Sigma_0\cap H_+(x,s))\subset\overline\Omega_0$ --- i.e., $\Sigma_0$ does not double back across $P_x(s)$. Equivalently, on the outward-normal geodesic ray from $y_*$, $\Sigma_0$ meets the ray in at most one point. As $p_0$ varies over $\Sigma_0$, the projection $\pi_{\mathcal C(s)}$ is therefore injective on $\Sigma_0$, exhibiting $\Sigma_0$ as a graph (in exponential coordinates from $\mathcal C(s)$) over $\pi_{\mathcal C(s)}(\Sigma_0)$. The image $\pi_{\mathcal C(s)}(\Sigma_0)$ is generally a proper subset of $\mathcal C(s)$; under the additional hypothesis that $\Sigma_0$ encloses $\mathcal C(s)$ (i.e., $\Sigma_0\subset\mathcal C^+(s)$), properness of $\Sigma_0$ with $\partial_\infty\Sigma_0=\{\mathbf 0,\infty\}$ implies $\pi_{\mathcal C(s)}(\Sigma_0)=\mathcal C(s)$, yielding a graph over the full cylinder.
\item[(b)] If there exists $s_+>0$ such that $\Sigma_0$ is contained in the closed cylindrical region bounded by $\mathcal{C}(s_+)$ (with $\partial_\infty\Sigma_0=\{\mathbf{0},\infty\}\subset\partial_\infty\mathcal{C}(s_+)$), then $\bar{s}\le s_+<\infty$.
\end{itemize}
In the remainder of this section we always work under~(b); in particular $\bar{s}$ is finite.

\begin{example}[Rotational examples for cases (a) and (b)]\label{ex:catenoid}
Recall that the equidistant tubes about $\gamma$ have mean curvature $H_{\mathcal C(\rho)}=(n-1)\coth\rho+\tanh\rho$, strictly decreasing from $+\infty$ to the
horospherical value $n$ as $\rho\to+\infty$. For a rotation hypersurface of constant
mean curvature $H$ about $\gamma$ (classified by do Carmo--Dajczer~\cite{doCarmoDajczer1983};
in $\H^3$ originally Mori~\cite{Mori1981}), whether case~(b) of Definition~\ref{admis tp}
holds is governed by the position of $H$ relative to this threshold value $n$.

\emph{(i) Case~(b) holds, and $\bar s$ has no uniform upper bound.} For $H>n$ the embedded (unduloid) members of the Delaunay-type family are complete,
properly embedded, periodic along $\gamma$, with $\partial_\infty \Sigma=\{\mathbf 0,\infty\}$, and
their hyperbolic distance to $\gamma$ oscillates in a bounded interval
$[\rho_{\min},\rho_{\max}]$; equivalently they are \emph{$1$-cylindrically bounded},
$\Sigma\subset\overline{\mathcal C^-(\rho_{\max})}$ ~\cite{Hsiang1982}. Hence case~(b)
of Definition~\ref{admis tp} applies and $\bar s(\Sigma)\le\rho_{\max}$. The constant
(isoparametric) members of this family are the equidistant cylinders $\mathcal C(s)$
themselves, for which $\mathcal C(s)\subset\overline{\mathcal C^-(s)}$ and the tangent
hyperplane $P_x(s)$ touches $\mathcal C(s)$ along $\gamma_x(s)$, so the bound is attained:
$\bar s(\mathcal C(s))=s$. As $s$ ranges over $(0,+\infty)$ the value $\bar s$ is therefore
unbounded above (and $\rho_{\max}\to+\infty$ as $H\downarrow n$ along the unduloid family).
This justifies the codomain $\bar s\in(0,+\infty]$ in Definition~\ref{admis tp}: no uniform
upper bound on $\bar s$ holds across two-ended properly embedded hypersurfaces with
$\partial_\infty \Sigma=\{\mathbf 0,\infty\}$.

\emph{(ii) Case~(b) is a genuine restriction.}
At the threshold $H=n$, the catenoid cousins of Bryant~\cite{Bryant1987} (see also
Umehara--Yamada~\cite{UmeharaYamada1993} and Rossman--Sato~\cite{RossmanSato1998}) are
properly embedded CMC surfaces of revolution about $\gamma$ with
$\partial_\infty\Sigma_a=\{\mathbf 0,\infty\}$, forming a one-parameter family $\{\Sigma_a\}_{a>0}$
in which $a$ is the \emph{neck radius}, i.e.\ the \emph{minimum} hyperbolic distance from
$\gamma$ to $\Sigma_a$, attained on the equatorial $S^1$. Their distance to $\gamma$ ranges
over $[a,+\infty)$ and tends to $+\infty$ at both ends, each end converging to one of the two
ideal endpoints of $\gamma$ (and asymptotic to a catenoid-cousin end, not to a
horosphere~\cite{CollinHauswirthRosenberg2001}). In particular $\Sigma_a$ is \emph{not}
contained in any tube $\mathcal C(s_+)$, so case~(b) fails for every $a>0$. 
\end{example}

We consider the curvature flow starting from the properly embedded hypersurface \(\Sigma_0\), $\partial_\infty \Sigma _0 =\{{\bf 0},\infty\}$ as in \eqref{eveq-again} i.e., a one-parameter family 
\(\{\Sigma_t\}_{t\in[0,T)} \subset \mathbb{H}^{n+1}\) of properly embedded, connected \(C^2\) hypersurfaces satisfying
\[
\partial_\infty \Sigma_t = \{{\bf 0},\infty\} \qquad \text{for all } t \in [0,T),
\]
and evolving with normal speed \(F\) in the direction of the inward unit normal \(\eta_t\). Here, by inward we mean the unit normal pointing towards \(\Omega_t\), the open set $\Omega_t \subset \mathbb{H}^{n+1}\setminus \Sigma _t$ with boundary \(\partial\Omega_t = \Sigma_t\) and asymptotic boundary \(\partial_\infty \Omega_t = \{{\bf 0},\infty\}\).\\
 In this section, we will \emph{assume} the existence of a viscosity solution with initial set \(\Sigma_0\); that is, a proper function \(u\) solving \eqref{eq:levelset} such that
\[
\partial_\infty(\partial E_t) = \{{\bf 0},\infty\} = \partial_\infty(\partial \Omega_t).
\]
Under this assumption, the same reflection arguments yield graphical and containment properties analogous to those in the compact case. The two-point analogue of Lemma~\ref{lem:transfer-1pt} reads as follows.

\begin{lemma}[Transfer of the reflection scheme to the two-point case]\label{lem:transfer-2pt}
Let $\Sigma\subset\H^{n+1}$ be a properly embedded $C^2$ hypersurface with $\partial_\infty\Sigma=\{\mathbf 0,\infty\}$. Then $\partial_\infty\Sigma\cap\partial_\infty P_x(s)=\emptyset$ and $\Sigma\cap\overline{H_+(x,s)}$ is compact, for every $x\in\partial_\infty\H^{n+1}\setminus\{\mathbf 0,\infty\}$ and every $s\in(0,+\infty)$ (with the foliation~\eqref{Hupper2pts} of Section~\ref{AF-2pts}). In particular, the Aleksandrov reflection method can be applied with this foliation, the first contact point (if any) occurs at a finite point of $\H^{n+1}$, and with the admissibility notion of Definition~\ref{admis tp}, Lemma~\ref{lem: admissibility} and Theorem~\ref{thm: admissible for all t} hold verbatim.
\end{lemma}
\begin{theorem}\label{thm2pts}
Let \(\Sigma_t\subset\mathbb{H}^{n+1}\), \(t\in(0,T)\), be a family of properly embedded \(C^2\) hypersurfaces evolving by \eqref{eveq-again}, with
\(\partial_\infty\Sigma_t=\{{\bf 0},\infty\}\) and such that \(\Sigma_0\) lies inside the 
cylinder  \(\mathcal C(s_+)\), with \(\mathcal C(s_+)\) as defined in \eqref{cylinder}. Then, for
every $x\in\partial_{\infty}\H^{n+1}\setminus\{{\bf 0},\infty\}$ and \(t\in(0,T)\), the portion
\[
\Sigma_t\cap H_+(x,\bar{s})
\]
is a \(C^2\)-graph (in exponential coordinates) over \(P_{x}(\bar{s})\). Moreover, the gradient of this graph is locally bounded by a
constant that is independent of \(t\) and of the particular choice of the
curvature speed \(F\).

\end{theorem}
For the next theorem we set some notation. For $s>0$, define
\[
   \lambda_s\;:=\;\frac{1}{\sinh s},
\]
so that in the upper half-space model the cylinder $\mathcal{C}(s)$ is the Euclidean cone $\{y=\lambda_s|x|\}$. For $p\in\H^{n+1}$ we denote by $q:=\pi_{\partial_\infty}(p)\in\R^n\times\{0\}$ its vertical projection onto the asymptotic boundary; for a point $y_0\in\mathcal{C}(\bar{s})$ we let $a\in\R^n$ be such that $y_0=(a,\lambda_{\bar{s}}|a|)$. Let $N$ denote the outward unit normal to $\mathcal{C}(\bar{s})$.

\begin{theorem}\label{thm:grad2pts}
Under the hypotheses of Theorem~\ref{thm2pts}, for every $t\in[0,T)$ the portion $\Sigma_t\cap\mathcal{C}^+(\bar{s})$ is a $C^2$-graph over the cylinder $\mathcal{C}(\bar{s})$, given by \eqref{cylinder}, in exponential coordinates, namely
\begin{equation*}
    \mathcal V_t \;=\; \bigl\{\,(y,\exp_y(r_t(y)\,N(y)))\;:\;y\in U\subset\mathcal{C}(\bar{s})\,\bigr\},
\end{equation*}
where $r_t:U\to(\bar{s},\infty)$ is a $C^2$ function. Moreover, writing $p_0=\exp_{y_0}(r_0\,N(y_0))$ with $y_0=(a,\lambda_{\bar{s}}|a|)\in\mathcal{C}(\bar{s})$ and $q_0=\pi_{\partial_\infty}(p_0)$, the intrinsic gradient of $r_t$ on $\mathcal{C}(\bar{s})$ satisfies, at $y_0$,
\begin{equation*}
    |\nabla r_t(y_0)|\;\leq\;\frac{\lambda_{\bar{s}}\sqrt{1+\lambda_{\bar{s}}}\,|a|}{r_0\,\lambda_{\bar{s}+r_0}\,|q_0|}.
\end{equation*}
\end{theorem}
\begin{proof}
 Fix \(t\in[0,T)\) and let \(p_0\in\Sigma_t\) be such that there exist a geodesic $\gamma_{x_0}$ passing through $p_0$ and perpendicular to the vertical geodesic $\gamma$ with $\gamma_{x_0}(-\infty)=x_0$ and $y_0\in\mathcal{C}(\bar{s})$ be such that $p_0=\exp_{y_0}(r_t(y_0)N(y_0))$ and also let $q_0=\pi_{\partial_\infty\H^{n+1}}(p_0)$ be the orthogonal projection of $p_0$ on $\R^n\times\{0\}.$\\

 Next for any $x\in\partial_\infty\H^{n+1}\setminus\{{\bf 0},\infty\}$, we write $P_x(\bar{s})$ as a Euclidean hemisphere centered at $c=(a(1+\lambda_{\bar{s}}^2),0)$ with radius $R=\lambda_{\bar{s}}\sqrt{1+\lambda_{\bar{s}}}\,|a|$, which is tangent to the Euclidean cone $\mathcal{C}(\bar{s})$ at $y=(a,\lambda_{\bar{s}}|a|)$ (recall the notation introduced before the statement).\\
 Then we have \(p_0\in H_+(x,\bar{s})\) whenever $|c-q_0|<\lambda_{\bar{s}}\sqrt{1+\lambda_{\bar{s}}}|a|$.
\begin{figure}[!ht]
    \centering
    \includegraphics[width=1\linewidth]{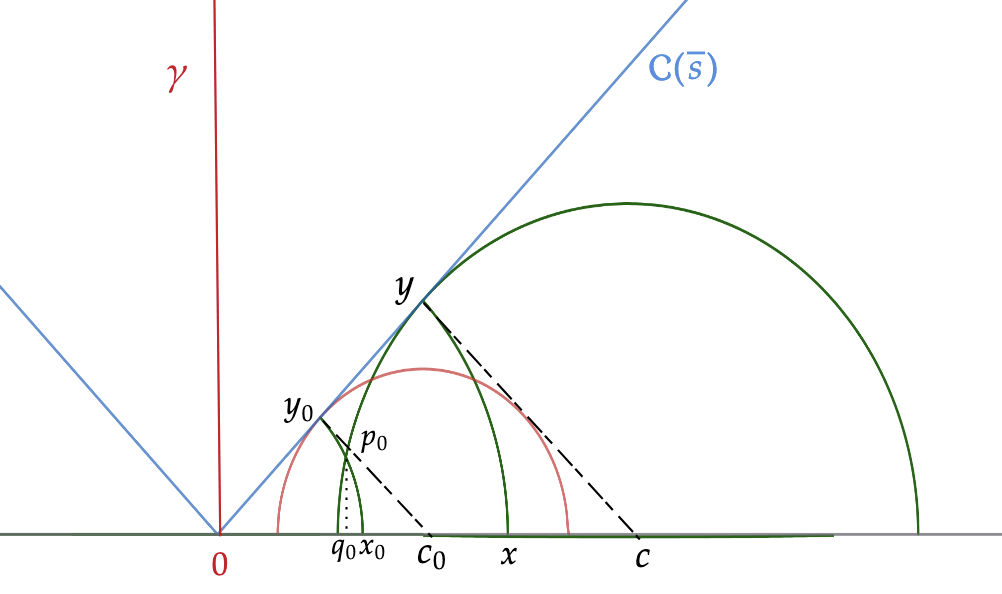}
    \caption{Gradient bound for  $\Sigma_t$}
    \label{fig:grad2pts}
\end{figure}

Since \(\bar{s}\ge s_0(x_0)\), there exists a neighborhood
\(\mathcal{V}\subset\Sigma_t\) of \(p_0\) which is a graph over a neighborhood
\(D\subset P_{x_0}(\bar{s})\). Shrinking \(\mathcal{V}\) if necessary, we may
assume that \(\mathcal{V}\) is a graph over a neighborhood
$U(y_0)\subset\mathcal{C}(\bar{s})$: since the tangent hemisphere $P_x(\bar{s})$ can be slid along $\mathcal{C}(\bar{s})$, varying the contact point while keeping $\mathcal R_{x,\bar{s}}$-admissibility, a neighbourhood of $p_0$ in $\Sigma_t$ can be represented as a graph over a neighbourhood of $y_0$ in $\mathcal{C}(\bar{s})$, namely
\[
\mathcal{V}=\big\{\big(y,\exp_y(r_t(y)N(y))\big):y\in U(y_0)\big\}
\]
for some \(C^2\) function \(r_t:U(y_0)\to(\bar{s},\infty)\) with
\(r_t(y_0)=r_0\).\\
Next we prove the gradient estimate. Let \(x\in\partial_\infty\H^{n+1}\setminus\{{\bf 0},\infty\}\) be such that
$|c-q_0|<\lambda_{\bar{s}}\sqrt{1+\lambda_{\bar{s}}}|a|$.
Then the boundary of the upper
half-space \(P_x(\bar{s})\) can be written as a Euclidean graph over
\(\R^n\times\{0\}\) of the form
\[
y(x)=\sqrt{R^2-|x-c|^2}.
\]
A direct computation gives
\[
\nabla _{\mathbb{R}^{n}}y(x)=\frac{x-c}{\sqrt{R^2-|x-c|^2}}
=\frac{x-c}{y(x)},
\]
and hence
\begin{equation*}
    |\nabla_{\mathbb{R}^{n}}y| (q_0)
=\frac{|q_0-c|}{y(q_0)}\leq \frac{\lambda_{\bar{s}}\sqrt{1+\lambda_{\bar{s}}}|a|}{\lambda_{\bar{s}+r_0}|q_0|}
\end{equation*}
here we used the fact that $p_0\in \mathcal{C}(\bar{s}+r_0)$ so $p_0=(q_0,\lambda_{\bar{s}+r_0}|q_0|)$ and $|p_0-q_0|=\lambda_{\bar{s}+r_0}|q_0|$ and hence $y(q_0)=\lambda_{\bar{s}+r_0}|q_0|.$

Let $e\in T_{y_0}\mathcal{C}(\bar{s})$ and $y:(-\epsilon,\epsilon)\to\mathcal{C}(\bar{s})$ be a curve on $\mathcal{C}(\bar{s})$ with $y(0)=y_0$ and $y'(0)=e$. Then
\[
\frac{d}{ds}\bigg|_{s=0}\exp_{y(s)}\bigl(r_t(y(s))\,N(y(s))\bigr)\;=\;d(\exp_{y_0})_{r_0N(y_0)}(\tau)\;\in\; T_{p_0}\Si_t,
\]
where $\tau=d(r_t)_{y_0}(e)\,N(y_0)+r_0\,dN_{y_0}(e)+e\in T_{y_0}\H^{n+1}$. Using Gauss's lemma applied to the hyperbolic exponential map at $y_0$, and the identity $N(p_0)=d(\exp_{y_0})_{r_0N(y_0)}(r_0N(y_0))$, we have
\begin{align*}
\Big\langle d(\exp_{y_0})_{r_0N(y_0)}(\tau),\;d(\exp_{y_0})_{r_0N(y_0)}(r_0N(y_0))\Big\rangle\;&=\;\langle\tau,r_0N(y_0)\rangle\\
    \;&=\;r_0\,d(r_t)_{y_0}(e).
\end{align*}
Note that, since $dN_{y_0}\colon T_{y_0}\mathcal{C}(\bar{s})\to T_{y_0}\mathcal{C}(\bar{s})$ is (the negative of) the shape operator of $\mathcal{C}(\bar{s})$ at $y_0$ and $e\in T_{y_0}\mathcal{C}(\bar{s})$, the vector $r_0\,dN_{y_0}(e)+e$ lies in $T_{y_0}\mathcal{C}(\bar{s})$ and is therefore orthogonal to $N(y_0)$ in the hyperbolic metric; the identity above is the full inner-product computation in the splitting $T_{y_0}\H^{n+1}=T_{y_0}\mathcal{C}(\bar{s})\oplus\R\,N(y_0)$.

Since a neighbourhood of $p_0$ in $\Sigma_t$ lies inside $H_+(x,\bar{s})$ and $P_x(\bar{s})$ is tangent to $\mathcal{C}(\bar{s})$ at $y_0$, pulling back via the normal exponential map (and using the shape-operator identity above to translate the comparison from the ambient Euclidean gradient into the intrinsic gradient on $\mathcal{C}(\bar{s})$), the graph of $r_t$ over $U(y_0)$ lies below the corresponding graph of $y(\cdot)$ and is tangent to it at $y_0$. We thus obtain
\[
|\nabla r_t(y_0)|\;\leq\;\frac{1}{r_0}\,|\nabla_{\mathbb{R}^{n}}y(q_0)|\;\leq\;\frac{\lambda_{\bar{s}}\sqrt{1+\lambda_{\bar{s}}}\,|a|}{r_0\,\lambda_{\bar{s}+r_0}\,|q_0|}.
\]
\end{proof}

\medskip

\subsection{Inverse curvature flows}\label{sec:icf-2pts}

The next result is the cylindrical counterpart of Theorem~\ref{thm:upbar}: it shows that if the initial hypersurface lies outside a hyperbolic cylinder, then the evolving hypersurfaces remain outside an explicit family of expanding cylinders for all later times. We keep the same structural assumptions on the curvature function $\mathcal F$ as in Section~\ref{sec:icf} (Gerhardt's setting~\cite{Gerhardt2014}), which in particular include IMCF; we also recall that the cylinder $\mathcal C(\rho)$ is isoparametric and hence evolves under the inverse curvature flow as a one-parameter family of cylinders $\{\mathcal C(\rho(t))\}$ with $\dot\rho(t)=1/H_{\mathcal C(\rho(t))}=1/((n-1)\coth\rho(t)+\tanh\rho(t))$ (cf.~\cite{alencar2026inversemeancurvatureflow}). In contrast with horospheres, equidistants and geodesic spheres, this ODE does not admit a closed-form integral, but $\rho(t)\to\infty$ linearly as $t\to\infty$, with $H_{\mathcal C(\rho(t))}\to n$.

\begin{theorem}[Cylinder as a lower barrier]\label{thm:cylbar}
Let $\{\Sigma_t\}_{t\in[0,T_{\max})}$ be a solution to the inverse curvature flow~\eqref{ICF} with $\partial_\infty\Sigma_t=\{\mathbf 0,\infty\}$ for every $t\in[0,T_{\max})$. Suppose that, for some $s_->0$, $\Sigma_0\subset\mathcal C^+(s_-)$, where $\mathcal C^+(s_-)$ is the mean-concave side of $\mathcal C(s_-)$ given by \eqref{cylinder}, that is
\[
d_{\H^{n+1}}(\gamma,p)\ge s_-
\quad\forall\,p\in\Sigma_0\,.
\]
 Then, for every $t\in[0,T_{\max})$,
\(
\Sigma_t\subset\mathcal C^+(s_-(t)),
\)
where $s_-(t)$ is defined by
\begin{equation}\label{eq:cylbar-evolution}
\sinh s_-(t)=e^{t/n}\sinh s_-.
\end{equation}
In particular,
\[
d_{\H^{n+1}}(\gamma,\Sigma_t)\ge s_-(t)>0
\]
for every $t\in[0,T_{\max})$, and $s_-(t)\to+\infty$ whenever $T_{\max}=+\infty$.
\end{theorem}

\begin{proof}
For each $x\in\gamma$, let $\mathcal S_{s_-}(x)$ denote the geodesic sphere of hyperbolic radius $s_-$ centered at $x$.

Since $\Sigma_0\subset\mathcal C^+(s_-)$, the sphere $\mathcal S_{s_-}(x)$ is disjoint from $\Sigma_0$. Under~\eqref{ICF}, the sphere $\mathcal S_{s_-}(x)$ evolves to spheres $\{\mathcal S_{s_-(t)}(x)\}_{t\ge 0}$ with the same center $x$ and radii $s_-(t)$ satisfying $\sinh s_-(t)=e^{t/n}\sinh s_-$ (cf.~\eqref{sphere evolution}). The avoidance principle for~\eqref{ICF} between the compact smooth flow $\{\mathcal S_{s_-(t)}(x)\}$ and the properly embedded flow $\{\Sigma_t\}$ yields
\[
B_{s_-(t)}(x)\cap\Sigma_t=\emptyset
\qquad
\forall\,x\in\gamma,
\quad
\forall\,t\in[0,T_{\max}).
\]
Taking the union over all $x\in\gamma$ gives
\[
\Sigma_t\cap
\left(\bigcup_{x\in\gamma} B_{s_-(t)}(x)\right)
=\emptyset.
\]
Since $
\bigcup_{x\in\gamma} B_{s_-(t)}(x)
=
\mathcal C^-(s_-(t)),
$
we conclude that
\(
\Sigma_t\subset\mathcal C^+(s_-(t))
\). Therefore
\[
d_{\H^{n+1}}(\gamma,\Sigma_t)\ge s_-(t)
=
\operatorname{arcsinh}\!\bigl(e^{t/n}\sinh s_-\bigr).
\]
This proves the theorem.
\end{proof}

\medskip

Combining Theorem~\ref{thm:cylbar} with Theorem~\ref{thm:grad2pts} yields the cylindrical analogue of the graphicality conclusion drawn from Theorem~\ref{thm:upbar} and Theorem~\ref{thm:gradup} in Section~4:

\begin{corollary}\label{cor:graph-after-T}
Under the hypotheses of Theorem~\ref{thm2pts} and assuming additionally that $\Sigma_0\subset\mathcal C^+(s_-)$ for some $s_->0$, define
\[
T\;:=\;n\,\log\!\Bigl(\frac{\sinh\bar s}{\sinh s_-}\Bigr)\,.
\]
For every $t\ge T$, $\Sigma_t$ is a global $C^2$-graph over the cylinder $\mathcal C(\bar s)$, with the uniform gradient bound of Theorem~\ref{thm:grad2pts}.
\end{corollary}

\begin{proof}
By Theorem~\ref{thm:cylbar}, $s_-(t)=\operatorname{arcsinh}(e^{t/n}\sinh s_-)\ge\bar s$ exactly when $t\ge T$. For such $t$, $\Sigma_t\subset\mathcal C^+(s_-(t))\subset\mathcal C^+(\bar s)$, so $\Sigma_t\cap\mathcal C^+(\bar s)=\Sigma_t$; the graphicality and gradient bound then follow from Theorem~\ref{thm:grad2pts}.
\end{proof}

\medskip

Specialising to the inverse mean curvature flow, the cylinder lower barrier and the graphicality corollary combine to give a partial counterpart of Theorem~\ref{thm:NonCompact}. We emphasise that the full analogue does \emph{not} transfer to the two-point setting:~\cite[Theorem~1]{Allen} is stated for hypersurfaces represented as graphs over the (totally umbilic) horosphere $\R^n\times\{0\}$ in the upper half-space model, and its proof relies on an ODE maximum principle at infinity whose forcing term is autonomous in the angle function $v y^{-1}$, a feature tied to the umbilicity of horospheres. In the two-point case the natural model surface is the cylinder $\mathcal C(\bar s)$, which has two distinct principal curvatures $\coth\bar s$ and $\tanh\bar s$; the analogous ODE forcing term ceases to be autonomous, and the cylinder itself is not stationary under IMCF (its hyperbolic distance from $\gamma$ grows unboundedly with~$t$). We therefore confine ourselves to the geometric control given by Theorems~\ref{thm:cylbar} and~\ref{thm:grad2pts}, and record long-time existence as an open problem.

\begin{theorem}[Non-compact IMCF with two points at infinity]\label{thm:NonCompact2pts}
Let $\Sigma_0\subset\H^{n+1}$ be a smooth, properly embedded hypersurface with $\partial_\infty\Sigma_0=\{\mathbf 0,\infty\}$ and $\Sigma_0\subset\mathcal C^+(s_-)\cap\mathcal C^-(s_+)$ for some $0<s_-<s_+<\infty$. Let $T_{\max}\in(0,+\infty]$ denote the maximal time of classical existence of the inverse mean curvature flow starting at $\Sigma_0$. Then:
\begin{itemize}
   \item[(a)] (\emph{cylinder lower barrier}) By Theorem~\ref{thm:cylbar}, $\Sigma_t\subset\mathcal C^+(s_-(t))$ for every $t\in[0,T_{\max})$, with $\sinh s_-(t)=e^{t/n}\sinh s_-$; in particular, $d_{\H^{n+1}}(\gamma,\Sigma_t)\ge s_-(t)>0$ uniformly on $[0,T_{\max})$, and $s_-(t)\to+\infty$ whenever $T_{\max}=+\infty$.
   \item[(b)] (\emph{graphicality and gradient bound}) For every $$t\ge T:=n\log(\sinh\bar s/\sinh s_-),$$ Corollary~\ref{cor:graph-after-T} represents $\Sigma_t$ as a global $C^2$-graph over the cylinder $\mathcal C(\bar s)$, with the uniform gradient bound of Theorem~\ref{thm:grad2pts}.
\end{itemize}
\end{theorem}

\begin{remark}\label{rem:open-2pt}
Long-time existence of the flow ($T_{\max}=+\infty$) and the existence of a non-trivial asymptotic profile remain open in the two-point setting. A natural strategy would be to combine the uniform gradient bound of Theorem~\ref{thm:grad2pts} with higher-regularity estimates in the spirit of Allen~\cite{Allen} or Scheuer~\cite{scheuer2015gradient,scheuer2015nonscale} for warped-product ambient spaces, but the cylindrical degeneration of the level-set PDE at the axis $\gamma$ prevents a direct adaptation. 
\end{remark}

\subsection*{Acknowledgements}
We would like to thank Brian Harvie and Julian Scheuer for useful conversations regarding the literature and possible extensions of this work.

Theodora Bourni and Aakash Mishra were supported by the grant NSF 2405007.

Jos\'e M. Espinar is partially supported by the Spanish MIC Grant PID2024-160586NB-I00, and the ``Maria de Maeztu'' Excellence Unit IMAG, reference CEX2020-001105-M, funded by\\
MCINN/AEI/10.13039/501100011033/CEX2020-001105-M.

    \bibliography{sources}
    \bibliographystyle{plain}

\end{document}